\documentclass[hidelinks,11pt,leqno]{article}
\usepackage[doc]{optional}
\usepackage[margin=0.9in]{geometry}
\usepackage{hyperref}
\usepackage{color}
\usepackage{float}
\usepackage{soul}
\usepackage{graphicx}
\usepackage{caption}
\usepackage{subcaption}
\usepackage[overload]{empheq}
\usepackage{bold-extra}

\usepackage{amsmath, amsthm, amssymb}

\usepackage{multirow}
\usepackage{enumitem}
\usepackage{empheq}

\usepackage{fancyhdr}
\fancyheadoffset[R]{0.5cm}
\newcommand{\menge}[2]{\big\{{#1}~\big |~{#2}\big\}}

\newcommand{\fenv}[1]%
{\ensuremath{\,\overrightarrow{\operatorname{env}}_{#1}}}
\newcommand{\benv}[1]%
{\ensuremath{\,\overleftarrow{\operatorname{env}}_{#1}}}

\newcommand{\scal}[2]{\left\langle{#1},{#2}  \right\rangle}

\newcommand{\RR}{\ensuremath{\mathbb R}}
\newcommand{\RP}{\ensuremath{\mathbb{R}_+}}
\newcommand{\RPP}{\ensuremath{\mathbb{R}_{++}}}

\newcommand{\RX}{\ensuremath{\,\left(-\infty,+\infty\right]}}

\newcommand{\NN}{\ensuremath{\mathbb N}}

\newcommand{\argmin}{\ensuremath{\operatorname*{argmin}}}

\newcommand{\prox}{\ensuremath{\operatorname{Prox}}}
\newcommand{\proj}{\ensuremath{\operatorname{P}}}

\newcommand{\conv}{\ensuremath{\operatorname{conv}}}

\newcommand{\Id}{\ensuremath{\operatorname{Id}}}

\newcommand{\sgn}{\ensuremath{\operatorname{sgn}}}

\newcommand{\bx}{\ensuremath{\mathbf{x}}}
\newcommand{\bX}{\ensuremath{{\mathbf{X}}}}

\newcommand{\by}{\ensuremath{\mathbf{y}}}

{\begin{list}{}{%
\settowidth{\labelwidth}{\textrm{#1~}}%
\setlength{\leftmargin}{\labelwidth+\labelsep}}}
{\end{list}}

\makeatletter
\def\namedlabel#1#2{\begingroup
   \def\@currentlabel{#2}%
   \label{#1}\endgroup
}
\makeatother

\makeatletter
\def\th@plain{%
  \thm@notefont{}
  \itshape 
}
\def\th@definition{%
  \thm@notefont{}
  \normalfont 
}
\makeatother

\newtheorem{theorem}{Theorem}[section]
\newtheorem{lemma}[theorem]{Lemma}

\theoremstyle{definition}
\newtheorem{definition}[theorem]{Definition}
\theoremstyle{definition}

\theoremstyle{definition}

\theoremstyle{definition}

\theoremstyle{definition}
\newtheorem{example}[theorem]{Example}
\theoremstyle{definition}
\newtheorem{remark}[theorem]{Remark}

\def\proof{\noindent{\it Proof}. \ignorespaces}
\def\endproof{\ensuremath{\hfill \quad \square}}

\allowdisplaybreaks 

\def\ox{\overline{x}}

\def\kkk{{k\in\NN}}

\def\disp{\displaystyle}
\def\dd{\delta}

\newcommand{\bee}{\ensuremath{{\mathbf{e}}}}

\newcommand{\wh}[1]{\widehat{#1}}

\newcommand{\matrx}[1]{\left(\begin{matrix}#1\end{matrix}\right)}

\begin{document}
\setlength\parindent{15pt}
\setlength\parskip{6pt}

\title{Optimization of Triangular Networks with Spatial Constraints}

\author{
Valentin R.\ Koch\footnote{
Design and Creation Products (DCP), Autodesk, Inc. Email: \texttt{valentin.koch@autodesk.com}}
~and~
Hung M.\ Phan
\footnote{
Department of Mathematical Sciences, Kennedy College of Sciences, University of Massachusetts Lowell, MA~01854, USA.
E-mail: \texttt{hung\_phan@uml.edu}}
}

\date{April 03, 2019}

\maketitle

\begin{abstract}
A common representation of a three dimensional object in computer applications, such as graphics and design, is in the form of a triangular mesh. In many instances, individual or groups of triangles in such representation need to satisfy spatial constraints that are imposed either by observation from the real world, or by concrete design specifications of the object. As these problems tend to be of large scale, choosing a mathematical optimization approach can be particularly challenging. In this paper, we model various geometric constraints as convex sets in Euclidean spaces, and find the corresponding projections in closed forms. We also present an interesting idea to successfully maneuver around some important nonconvex constraints while still preserving the intrinsic nature of the original design problem. We then use these constructions in modern first-order splitting methods to find optimal solutions.
\end{abstract}

  

\section{Introduction and Motivation}

The notation in the paper is fairly standard and follows largely \cite{BC2017}. $\RR$ denotes the set of real numbers. By ``$x:=y$", or equivalently ``$y=:x$", we mean that ``$x$ is defined by $y$". The {\em assignment operators} are denoted by ``$\longleftarrow $" and ``$\longrightarrow$". The angle between two vectors $\vec{n}$ and $\vec{q}$ is denoted by $\angle(\vec{n},\vec{q})$. The cross product of $\vec{n}_1$ and $\vec{n}_2$ in $\RR^3$ is denoted by $\vec{n}_1\times\vec{n}_2$.

\subsection{Abstract Problem Formulation}
\label{ss:abs_prob}

Throughout the paper, we assume that $n\in\{3,4,\ldots\}$ and $X=\RR^n$ with standard inner product $\scal{\cdot}{\cdot}$ and induced norm $\|\cdot\|$. Assume that $G=(V_0,E_0)$ is a given {\em triangular mesh} on $\RR^2$ where $V_0$ is the set of vertices and $E_0$ is the set of (undirected) edges, i.e.,
\begin{align*}
V_0&:=\menge{p_i:=(p_{i1},p_{i2})\in\RR^2}
{i\in\{1,2,\ldots,n\}},\\
E_0&\subseteq\menge{ p_ip_j\equiv p_jp_i}{ p_i,p_j\in V_0}.
\end{align*}
From $G$, we can derive the set of triangles of the mesh as follows
\begin{equation*}
T_0:=\menge{\Delta=(p_ip_jp_k)}{\{p_ip_j,p_jp_k,p_kp_i\}\subseteq E_0}.
\end{equation*}
We first aim to
\begin{subequations}\label{e:feas}
\begin{equation}
\text{find a vector}\quad z=(z_1,\ldots,z_n)\in X
\end{equation}
such that the points
\begin{equation}
\{P_i:=(p_{i1},p_{i2},z_i)\}_{i\in\{1,\ldots,n\}}
\quad\text{satisfy a given set of constraints.}
\end{equation}
\end{subequations}
Clearly, the points $\{P_i\}_{i\in\{1,\ldots,n\}}$ also form a corresponding triangular mesh $S$ in three dimensions. Therefore, if there is no confusion, we will also use $E_0,V_0,T_0$ to denote the sets of vertices, edges, and triangles of the three dimensional mesh.

Several types of constraints for triangular meshes are listed below:
\begin{itemize}
\item {\bf interval constraints:} For a given subset $I$ of $\{1,2,\ldots,n\}$, for all $i\in I$, the entries $z_i$ must lie in a given interval of $\RR$.
\item {\bf edge slope constraints:} For a given subset $E$ of the edges $E_0$, and for every edge $P_iP_j\in E$, the slope
\begin{equation*}
s_{ij}:=\frac{z_i-z_j}{ \sqrt{(p_{i1}-p_{j1})^2+(p_{i2}-p_{j2})^2}}
\end{equation*}
must lie in a given subset of $\RR$.

\item {\bf edge alignment constraints:} For a given pair of edges $P_iP_j, P_mP_n \in E_0$, the edges must have the same slope $s_{ij} = s_{mn}$.

\item {\bf surface alignment constraints:} For a given pair of triangles $\Delta_i,\Delta_j \in T_0$ the \emph{normal vectors} $\vec{n}_{\Delta_i}$ and $\vec{n}_{\Delta_j}$ must be parallel.

\item {\bf surface orientation constraints:} For a given subset $T$ of the triangles $T_0$ and for each triangle $\Delta\in T$, there is a given set of vectors $Q_\Delta\subset\RR^3$ such that for each $\vec{q}\in Q_\Delta$, the {\em angle} between the {\em normal vector} $\vec{n}_\Delta$ and $\vec{q}$ must lie in a given subset $\theta_{\vec{q}}$ of $\left[0,\pi\right]$.
\end{itemize}

Suppose there are $J$ constraints imposed on a model. For $j\in\{1,2,\ldots,J\}$, we denote by $C_j$ the set of all points that satisfies the $j$-th constraint. Thus, \eqref{e:feas} can be written in the mathematical form
\begin{equation}\label{e:feas2}
\text{find a point}\quad
x\in C:=\bigcap_{j\in\{1,2,\ldots,J\}}C_j,
\end{equation}
which we refer to as {\em feasibility problem}. Moreover, of the infinitude of possible solutions for \eqref{e:feas2}, one may be particularly interested in those that are {\em optimal} in some sense. For instance, it could be desirable to find a solution that minimizes the slope change between adjacent triangles, a solution that minimizes the volume between the initial triangles and the triangles in the final solution, or variants and combinations thereof. If more than one objective function is of interest, it is common to additively combine these functions, perhaps by scaling the functions based on their different levels of importance. In summary, our goal is to solve the problem
\begin{equation}\label{e:fprob}
\min\quad{F(z)}
\quad\text{subject to}\quad
z\in C:=\bigcap_{j\in\{1,2,\ldots,J\}}C_j,
\end{equation}
where $F$ may be a sum of (scaled) objective functions. The function $F$ is typically nonsmooth which prevents the use of standard optimization methods.
 
\subsection{Computer-Aided Design for Architecture and Civil Engineering Structures}
\label{ss:civil}

The abstract problem formulation in Section~\ref{ss:abs_prob} has some concrete applications in computational surface generation of triangular meshes. In particular, in {Computer-Aided Design} (CAD), triangular meshes are widely used in various engineering disciplines. For example, in architecture and civil engineering, existing and finished ground surfaces are represented by triangulated irregular networks. Slopes are relevant in the context of drainage, in both, civil engineering (transportation structures), and architecture (roof designs), as well as in the context of surface alignments such as solar farms, embankment dams, and airport infrastructure layouts.

A concrete problem that arises in civil engineering design is the grading of a parking lot. Within a given area, the engineer has to define grading slopes such that the parking lot fits within existing structures, the drainage requirements on the lot are met, and safety and comfort is taken into account. Besides these requirements, the engineer would like to change the existing surface as little as possible, in order to save on earthwork costs.

\begin{figure}[!htb]
    \centering
    \begin{subfigure}[b]{0.3\textwidth}
        \includegraphics[width=\textwidth]{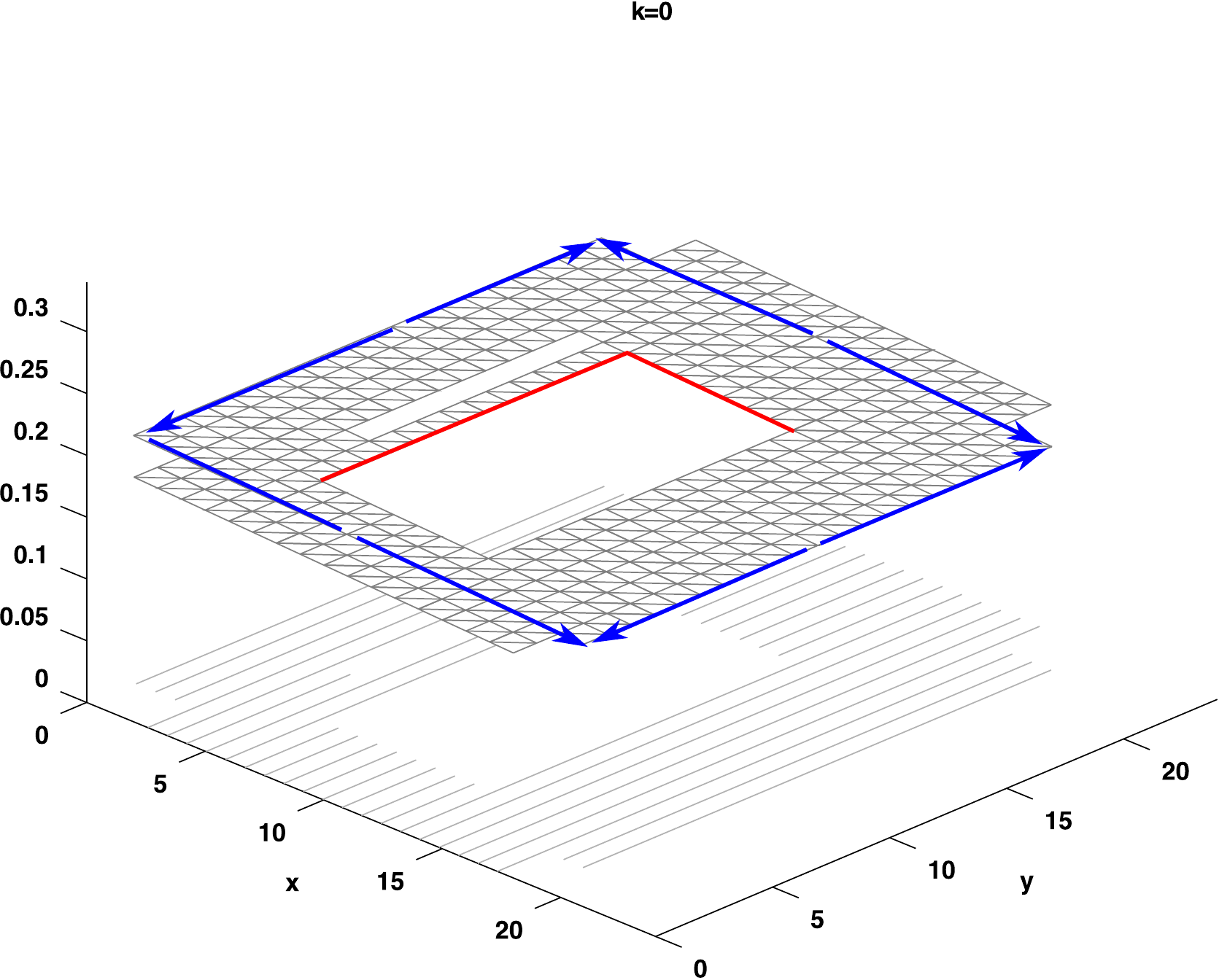}
        \caption{Parking A}
        \label{fig:parkingplanA}
    \end{subfigure}
    ~ 
    \begin{subfigure}[b]{0.3\textwidth}
        \includegraphics[width=\textwidth]{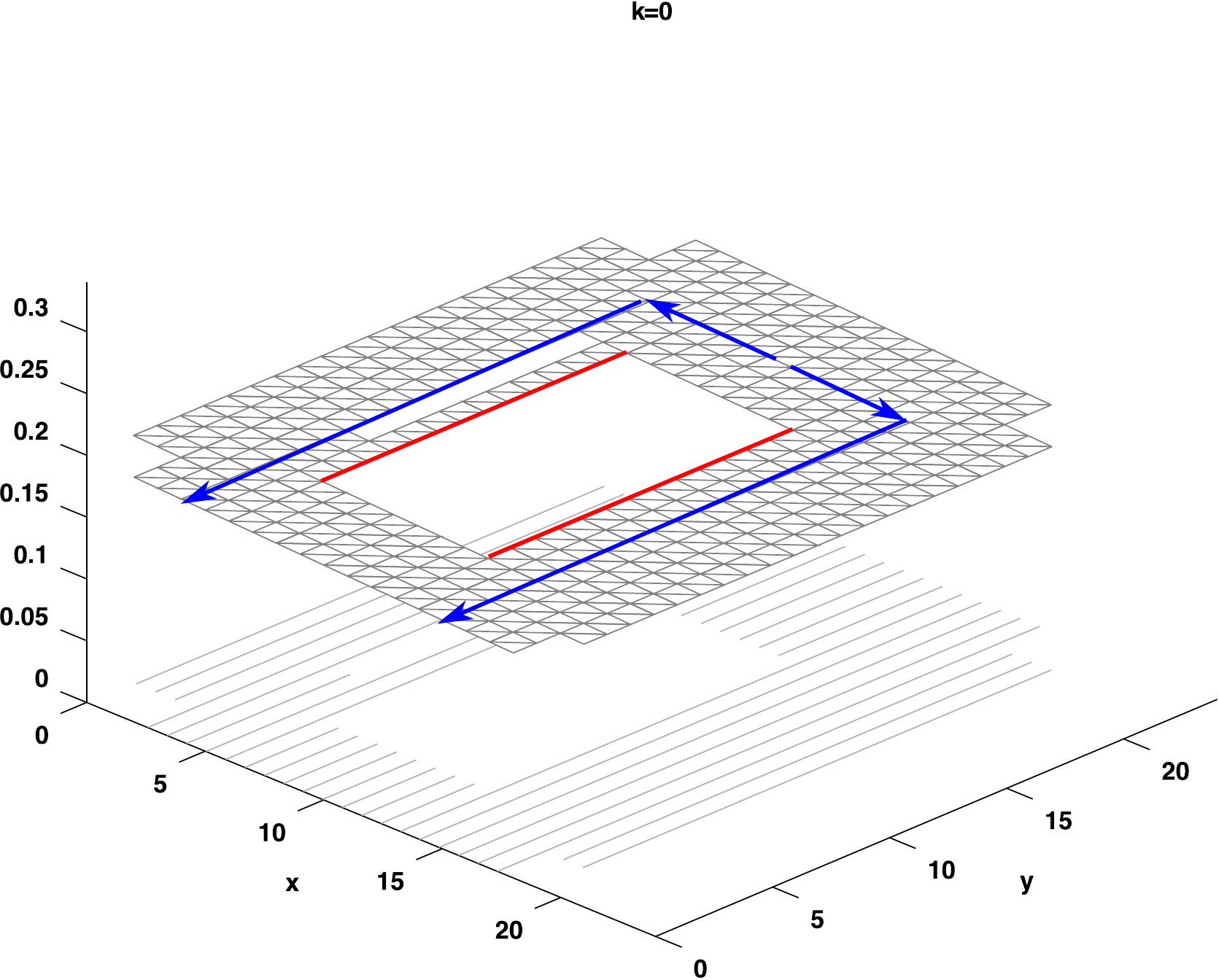}
        \caption{Parking B}
        \label{fig:parkingplanB}
    \end{subfigure}
    ~ 
    \begin{subfigure}[b]{0.3\textwidth}
        \includegraphics[width=\textwidth]{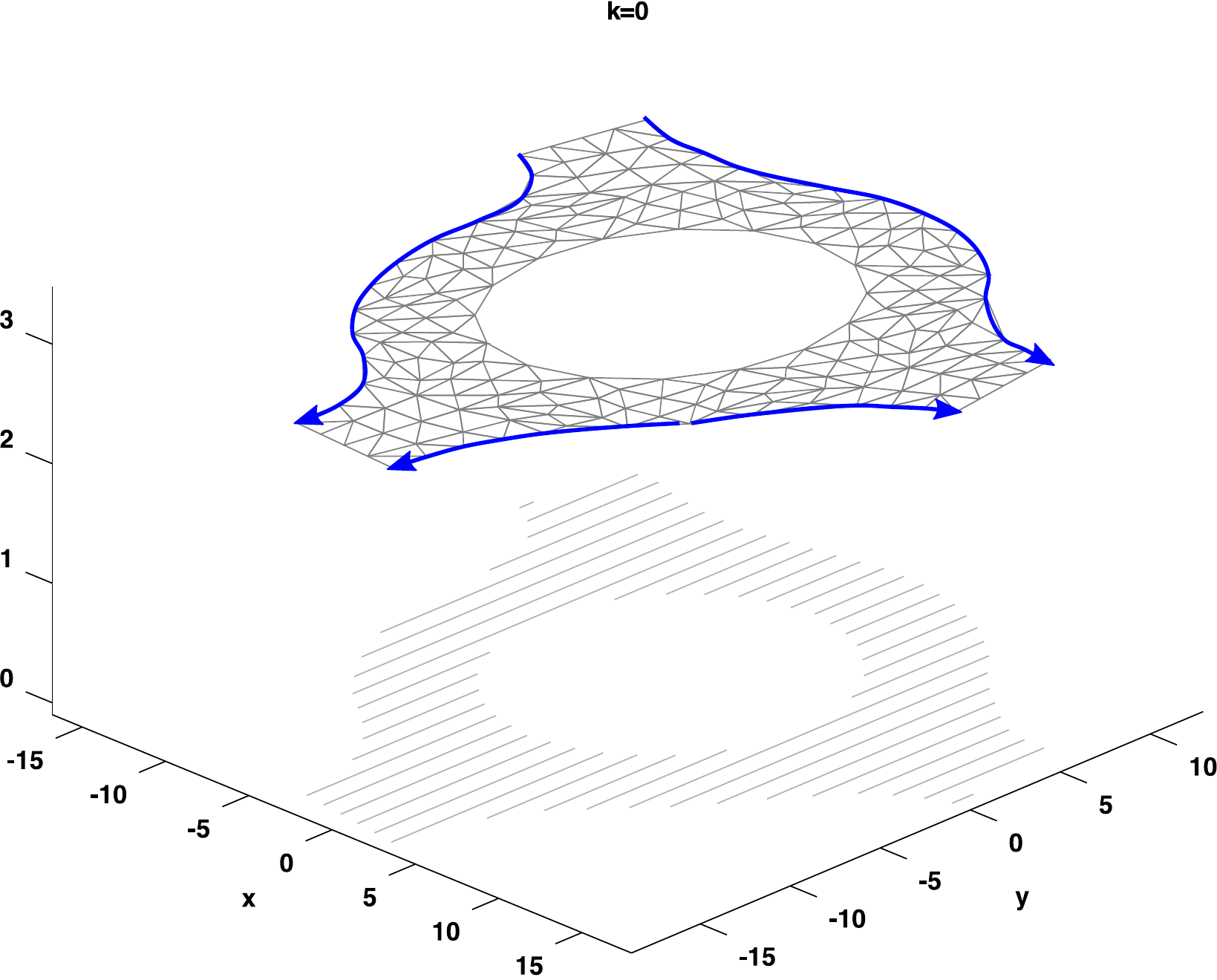}
        \caption{Roundabout}
        \label{fig:roundaboutplan}
    \end{subfigure}
    \caption{Drainage schemes for various engineering structures.}\label{f:parking}
\end{figure}

Consider Figure~\ref{f:parking}. The triangular mesh in Figure~\ref{fig:parkingplanA} represents the existing ground of a planned parking lot. The engineer would like the water to drain away from the building, and along the blue drain lines into the four corners. Red lines indicate where the triangle edges need to be aligned. In Figure~\ref{fig:parkingplanB}, the engineer would like to study a different scheme, where the drainage happens in parallel, on either side of the building. Lastly, Figure~\ref{fig:roundaboutplan} represents the triangular mesh of existing ground for a roundabout, where a minimum slope is required from the inner circle to the outer one, and water needs to drain from the road at the top along the outer circle to the roads on either side at the bottom.

\subsection{Objective and Outline of This Paper}

This paper aims to provide a framework for modeling practical design problems using geometric constraints in three dimensions. These problems can then be solved by iterative optimization methods. This involves the introduction and computation of projection/proximity operators of new constraints and objective functions. Once all required formulas are accomplished, they will be used in iterative optimization methods to obtain the solutions.

The paper is organized as follows.  Section~\ref{s:overvw} contains an overview of iterative methods that will be employed. From Section~\ref{s:linconstr} to Section~\ref{s:prj_minsl}, we derive the projection operators of the above spatial constraints in closed forms. Particularly in Section~\ref{s:prj_minsl}, we present an interesting idea to {\em modify} certain nonconvex constraints into convex ones that retain the intrinsic nature of the original design problem. To the best of our knowledge, this is the {\em first} methodology to successfully maneuver around such nonconvexity. Utilizing these constructions, we include in Section~\ref{s:optim} some optimization problems of interest. Finally, we present the numerical experiment in Section~\ref{s:experiment} and some remarks in Section~\ref{s:conclus}.

\section{Methods Overview}
\label{s:overvw}

\subsection{Projections onto Constraint Sets}

A constraint set $C\subseteq X$ is the collection of all feasible data points, i.e., points that satisfy some requirements. Suppose the given data point $z\in X$ is not feasible (i.e., $z\not\in C$), we aim to {\em modify} $z$ so that the newly obtained point $x$ is feasible (i.e., $x\in C$); and we would like to do it with {\em minimal} adjustment on $z$. This task can be achieved by using the projection onto $C$. Recall that the projection of $z$ onto $C$, denoted by $\proj_Cz$, is the solution of the optimization problem
\begin{equation*}
\proj_C z =\argmin_{x\in C}\|z-x\|=
\menge{x\in C}{\|z-x\|=\min_{y\in C} \|z-y\|}.
\end{equation*}
It is well known that if $C$ is nonempty, closed, and {\em convex}\footnote{$C$ is convex if for all $x,y\in C$ and $\lambda\in\left[0,1\right]$, we have $(1-\lambda)x+\lambda y\in C$}, then $\proj_C z$ is singleton, see, for example, \cite[Theorem~2.10]{rusz2006}.

\subsection{Proximity Operators}

Suppose $f:X\to\RX$ is a proper convex lower semicontinuous function\footnote{See, e.g., \cite{rock70} 
and \cite{BC2017} for relevant materials in convex analysis} and $x$ is a given point in $X$. Then it is well known (see~\cite[Section~12.4]{BC2017}) that the function
\begin{equation*}
X\to\RX:y\mapsto f(y)+\tfrac{1}{2}\|x-y\|^2
\end{equation*}
has a unique minimizer, which we will denote by $\prox_f(x)$. The induced operator $\prox_f:X\to X$
is called the {\em proximal mapping} or {\em proximity operator} of $f$ (see~\cite{moreau65}). Note that if $f$ is the {\em indicator function} of a set $C$ (the indicator function $\iota_C$ is defined by $\iota_C(x)=0$ if $x\in C$ and $\iota_C(x)=+\infty$ otherwise), then $\prox_f=\proj_C$. Thus, proximity operators are generalizations of projections.

\subsection{Iterative Methods for Optimization Problems}

Iterative optimization methods are often used for solving \eqref{e:fprob}, which may require the computations of proximity and projection operators for the functions and constraint sets involved.

It turns out that all spatial constraints encountered in our settings are convex and closed. Hence, their projections always exist and are unique. Moreover, we also successfully obtain explicit formulas for these projections. In the coming sections, we will make the formulas as convenient as possible for software implementation. As we will make proximity operators available for several types of objective functions, any iterative optimization methods that utilize proximity operators, for example, \cite{BCH2013,bricenocombettes2011,comb-pesq-08,lion-mercer-79}, can be used to solve the corresponding problems. Let us describe one such method, the Douglas--Rachford (DR) splitting algorithm. The DR algorithm emerged from the field of differential equations \cite{doug-rach-56}, and was later made widely applicable in other areas thanks to the seminal work~\cite{lion-mercer-79}.

To formulate DR algorithm, we first use indicator functions to convert \eqref{e:fprob} into
\begin{equation}
\label{e:fprob2}
\min\quad F(x)+\iota_{C_1}+\iota_{C_2}+\cdots+\iota_{C_J} 
\quad\text{subject to}\quad x\in X.
\end{equation}
So, it suffices to present DR for the following general optimization problem
\begin{equation}
\label{e:fprob2b}
\min\quad \sum_{i=1}^m f_i(x)
\quad\text{subject to}\quad x\in X,
\end{equation}
where each $f_i$ is a proper convex lower semicontinuous function on $X$. The DR operates in the product space
$\bX:=X^{m}$ with inner product
$\scal{\bx}{\by}:=\sum_{i=1}^m\scal{x_i}{y_i}$
for $\bx=(x_i)_{i=1}^m$ and $\by=(y_i)_{i=1}^m$. Set the starting point $\bx_0=(z,\ldots,z)\in\bX$, where $z\in X$. Given $\bx_k=(x_{k,1},\ldots,x_{k,m})\in\bX$, we compute
\begin{subequations}\label{e:dr}
\begin{align}
&\ox_k :=\frac{1}{m}\sum_{i=1}^m x_{k,i},\\
\forall i\in \{1,\ldots,m\}:\quad &x_{k+1,i}:=x_{k,i}-\ox_k+\prox_{\gamma f_i}(2\ox_{k}-x_{k,i}),\\
\text{then update}\quad
&\bx_{k+1}:=(x_{k+1,1},\ldots,x_{k+1,m}).
\end{align}
\end{subequations}
Then the sequence $(\ox_k)_\kkk$ converges to a solution of \eqref{e:fprob2b}. We note that $(\bx_k)_\kkk$ and $(\ox_k)_\kkk$ are referred to as {\em governing} and {\em monitored} sequences, respectively.

{It is worth mentioning that when all $f_i$'s are indicator functions of the constraints (thus, all proximity operators become projections), then \eqref{e:fprob2b} becomes a pure feasibility problem
\begin{equation}\label{e:intersec_prob}
\text{find a point}\quad
x\in C:=\bigcap_{j\in\{1,2,\ldots,J\}}C_j.
\end{equation}
Therefore, all optimization methods that work for \eqref{e:fprob2b} can also be used for \eqref{e:intersec_prob}. Nevertheless, it is worth mentioning that there are many iterative projection methods that are specifically designed for feasibility problems. At the first glance, it might be tempting to solve such problem by finding the projection onto the intersection $C$ directly. However, this is often not possible due to the complexity of $C$. A workaround is to utilize the projection $\proj_{C_j}$ onto each constraint, if the explicit formula is available. Then with an initial point, one can {\em iteratively execute} the projections $P_{C_j}$'s to derive a solution for \eqref{e:intersec_prob}. Let $z_0\in X$ be the initial data point. The following are two simplest instances among iterative projection methods (see \cite{BC2017,cegielskibook2012,cen_zen_par1997} and the references therein)
\begin{itemize}
\item {\bf cyclic projections:} Given $z_k$, we update $z_{k+1}:= T z_k$ where
$T:=\proj_{C_J}\proj_{C_{J-1}}\cdots \proj_{C_2}\proj_{C_1}$.
\item {\bf parallel projections:} Given $z_k$, we update $z_{k+1}:= T z_k$ where
$T:=\frac{1}{J}\big(\proj_{C_1}+\proj_{C_2}+\cdots +\proj_{C_J}\big)$.
\end{itemize}
In addition, when projection methods succeed (see~\cite{cen_chen_com_12} for some interesting examples), they have various attractive features: they are easy to understand, simple for implementation and maintenance, and sometime can be very fast. We refer the readers to \cite{BBsirev93,BK15,cegielskibook2012,cen_zen_par1997} for more comprehensive discussions on projection methods.}

\section{Projections onto Linear Constraints}
\label{s:linconstr}

By linear constraints, we refer to any constraint on the triangular mesh that can be represented by a system of linear equalities and inequalities. Indeed, this class includes several important constraints in design problems. In this section, we will analyze those constraints and their projectors.

\subsection{Interval Constraint}
\label{ss:prj_interp}

Similar to \cite[Section~2.2]{BK15}, we assume that $I$ is a subset of $\{1,2,\ldots,n\}$ and $(l_i)_{i\in I},(u_i)_{i\in I}\in\RR^I$ are given. Define
\begin{equation*}
Y:=\menge{z=(z_1,z_2,\ldots,z_n)\in X}{\forall i\in I:\ l_i\leq z_i\leq u_i}.
\end{equation*}
Then one can readily check that $Y$ is closed and convex. The following explicit formula is for the projection onto $Y$, whose proof is straightforward, thus, omitted.
\begin{align*}
\proj_Y&:X \to X:
(z_1,z_2,\ldots,z_n)\mapsto
(x_1,x_2,\ldots,x_n),\\
&\text{where}\ 
x_i=\begin{cases}
\max\{l_i,\min\{u_i,z_i\}\},& i\in I,\\
z_i,& i\in\{1,2,\ldots,n\}\smallsetminus I.
\end{cases}
\end{align*}

\subsection{Edge Minimum Slope Constraint}
Let $P_i=(p_{i1},p_{i2},z_i)$ and $P_j=(p_{j1},p_{j2},z_j)$.
The designer may require that the (directional) slope from $P_j$ to $P_i$ must be no less than a threshold level $s_{ij}$. More specifically, since the value $p_{i1},p_{i2},p_{j1},p_{j2}$ are fixed, we can write this constraint as
\begin{equation*}
z_{i}-z_{j}\geq \alpha_{ij}
:=s_{ij}\sqrt{(p_{i1}-p_{j1})^2+(p_{i2}-q_{j2})^2},
\end{equation*}
which we will call the {\em edge minimum slope constraint}. This constraint is a linear inequality, thus, convex. The projection formula onto this type of constraint can be derived analogously to \cite[Section~2.3]{BK15}.

\subsection{Low Point Constraint}
\begin{definition}[low point]
A point $P_j=(p_{j1},p_{j2},z_j)$ on the mesh is called a low point if each point $P_i=(p_{i1},p_{i2},z_i)$ connected to $P_j$ satisfies the edge minimum slope constraint
\begin{equation*}
z_i-z_j\geq \alpha_{i}
:=s_{i}\sqrt{(p_{i1}-p_{j1})^2+(p_{i2}-q_{j2})^2},
\end{equation*}
where all $s_{i}\in\RR$ are given.
\end{definition}

We can treat low point constraint as a {\em single} constraint even though it is the intersection of finitely many edge slope constraints. The following result shows how to project onto this constraint. First, we need a simple lemma

\begin{lemma}\label{l:181227c}
Let $a,b\in\RR$ and $k\geq 1$. Assume that
$ka\leq b+(k-1)\max\{a,b\}$. Then $a\leq b$.
\end{lemma}
\begin{proof}
By assumption, we must have either
$ka\leq b+(k-1)a$ or $ka\leq b+(k-1)b$, and any one of them readily implies $a\leq b$.
\end{proof}

\begin{theorem}[projection onto a low point constraint]
\label{t:lwpt}
Let $\alpha_2,
\ldots,\alpha_m\in\RR$. Define the set
\begin{equation*}
E:=\menge{(x_1,\ldots,x_m)\in\RR^m}{\forall i\in\{2,\ldots,m\}:\ x_i-x_1\geq\alpha_i}.
\end{equation*}
Let $z:=(z_1,\ldots,z_m)\in\RR^m$. Let $\dd_1:=z_1$ and let
$\dd_2\leq\dd_3\leq \cdots\leq \dd_m$
be the values $\{z_i-\alpha_i\}_{i\in\{2,\ldots,m\}}$ in nondecreasing order. Let $k$ be the largest number in $\{1,\ldots,m\}$ such that $\dd_k\leq(\dd_1+\cdots+\dd_k)/k$.
Then the projection $x:=(x_1,x_2,\ldots,x_m):=\proj_Ez$ is given by
\begin{equation*}
x_1=(\dd_1+\cdots+\dd_k)/k\quad \text{and}\quad
\forall i\in\{2,\ldots,m\}:\ 
x_i=\max\{x_1,z_i-\alpha_i\}+\alpha_i.
\end{equation*}
\end{theorem}
\proof
First, $x$ is the solution of
\begin{subequations}\label{e:161002a}
\begin{align}
\min\quad &(x_1-z_1)^2+\ldots+(x_m-z_m)^2\\
\text{s.t}\quad& x_1-x_i+\alpha_i\leq 0,\quad i\in\{2,\ldots,m\}.
\end{align}
\end{subequations}
Let $y:=(y_1,\ldots,y_m)$ where $y_1=x_1$ and $y_i:=x_i-\alpha_i$ for $i\in\{2,\ldots,m\}$. Without relabeling, we may assume $\dd_i=z_i-\alpha_i$ for all $i\in\{2,\ldots,m\}$. Then \eqref{e:161002a} becomes
\begin{subequations}\label{e:161002b}
\begin{align}
\min\quad &\varphi(y):=(y_1-\dd_1)^2+\ldots+(y_m-\dd_m)^2\\
\text{s.t}\quad& g_i(y):=y_1-y_i\leq 0,\quad i\in\{2,\ldots,m\}.
\end{align}
\end{subequations}
To find the (unique) solution, we use Lagrange multipliers: there exist $\lambda_2,\ldots,\lambda_m$, such that
\begin{subequations}\label{e:161002c}
\begin{align}
(1/2)\nabla \varphi(y)+\lambda_2\nabla g_2(y)+\ldots + \lambda_m\nabla g_m(y)&=0,\label{e:161002ca}\\
\forall i\in\{2,\ldots,m\}:\quad
\lambda_ig_i(y) =0,\quad
\lambda_i\geq 0,\quad
g_i(y)&\leq 0.\label{e:161002cb}
\end{align}
\end{subequations}
Then \eqref{e:161002ca} implies
\begin{align*}
(y_1-\dd_1)+\lambda_2+\cdots+\lambda_m&=0\\
\forall i\in\{2,\ldots,m\}:\quad (y_i-\dd_i)-\lambda_i &=0.
\end{align*}
So $y_1\leq \dd_1$ because $\lambda_2,\ldots,\lambda_m\geq 0$. By substitution, we get
\begin{equation}\label{e:161002d}
y_1+y_2+\ldots+y_m=\dd_1+\dd_2+\ldots +\dd_m.
\end{equation}
Also, \eqref{e:161002cb} reads as, for each $i\in\{2,\ldots,m\}$,
\begin{equation*}
0=\lambda_i g_i(y)=(y_i-\delta_i)(y_1-y_i),\quad
y_i-\delta_i\geq 0,\quad
\text{and}\quad
y_1-y_i\leq 0.
\end{equation*}
It follows that either $y_i=\delta_i\geq y_1$ or $y_1=y_i\geq\delta_i$, i.e.,
\begin{equation}\label{e:170107a}
\forall i\in\{2,\ldots,m\}:\quad y_i=\max\{y_1,\dd_i\}.
\end{equation}
Substituting \eqref{e:170107a} into \eqref{e:161002d} yields
\begin{equation}\label{e:161002e}
\dd_1+\cdots +\dd_m=
y_1+\max\{y_1,\dd_2\}+\cdots+\max\{y_1,\dd_m\}.
\end{equation}
So for all $j\in\{1,\ldots,m\}$, \eqref{e:161002e} implies
\begin{equation}\label{e:181227a}
\dd_1+\cdots +\dd_m\geq
\underbrace{y_1+\cdots+y_1}_{\text{$j$ terms}}+\dd_{j+1}+\cdots+\dd_m
\quad\Longrightarrow\quad
y_1\leq \frac{\dd_1+\cdots+\dd_j}{j}.
\end{equation}
Next, let $k$ be the largest number in $\{1,\ldots,m\}$ that satisfies
\begin{equation}\label{e:161002g}
\dd_k\leq\frac{\dd_1+\cdots+\dd_k}{k}.
\end{equation}
The number $k$ is well defined since at least \eqref{e:161002g} is true for $k=1$. Now we claim that
\begin{equation}\label{e:181227d}
y_1=\frac{\dd_1+\cdots+\dd_k}{k}
\end{equation}
by considering two cases.

{\it Case~1:} $k=m$. Using \eqref{e:161002g}, \eqref{e:161002e}, and $\dd_2\leq\dd_3\leq\cdots\leq\dd_m$, we have
\begin{align*}
m\dd_m&\leq\dd_1+\cdots+\dd_m
=y_1+\max\{y_1,\dd_2\}+\cdots+\max\{y_1,\dd_m\}\\
&\leq y_1+(m-1)\max\{y_1,\dd_m\}.
\end{align*}
Then Lemma~\ref{l:181227c} implies $\dd_m\leq y_1$. Using this in \eqref{e:161002e}, we derive $\dd_1+\cdots+\dd_m=my_1$ which is \eqref{e:181227d}.

{\it Case~2:} $k<m$. By the choice of $k$, we have $\frac{\dd_1+\cdots+\dd_{k+1}}{k+1}<\dd_{k+1}$, which implies
\begin{equation*}\label{e:181227b}
\frac{\dd_1+\cdots+\dd_{k}}{k}<\dd_{k+1}.
\end{equation*}
Combining with \eqref{e:181227a}, we conclude that $y_1<\dd_{k+1}$. Using $y_1<\dd_{k+1}$ and \eqref{e:161002g} in \eqref{e:161002e}, we obtain
\begin{align*}
k\dd_k+\dd_{k+1}+\cdots+\dd_m
&\leq(\dd_1+\cdots +\dd_k)+\dd_{k+1}+\cdots+\dd_m\\
&= y_1+\max\{y_1,\dd_2\}+\cdots+\max\{y_1,\dd_m\}\\
&= y_1+\underbrace{\max\{y_1,\dd_2\}+\cdots+\max\{y_1,\dd_k\}}_{\text{$k-1$ terms}}
+\dd_{k+1}+\cdots+\dd_m\\
&\leq y_1+(k-1)\max\{y_1,\dd_k\}+\dd_{k+1}+\cdots+\dd_m.
\end{align*}
This implies
\begin{equation*}
k\dd_k\leq y_1+(k-1)\max\{y_1,\dd_k\}.
\end{equation*}
Again, Lemma~\ref{l:181227c} implies $\dd_k\leq y_1$. Now using $\dd_k\leq y_1<\dd_{k+1}$ in \eqref{e:161002e}, we have
\begin{equation*}
\dd_1+\cdots +\dd_m=ky_1+\dd_{k+1}+\cdots+\dd_m,
\quad\text{which implies \eqref{e:181227d}}.
\end{equation*} 

So, \eqref{e:181227d} is true. Finally, we compute $y_i$'s from \eqref{e:170107a} and \eqref{e:181227d}, then use them to derive $x_i$'s.
\endproof

\subsection{Edge Alignment Constraint}

On the triangular mesh, the designer may want a {\em constant slope} on a particular path, in which case we say the path is ``aligned". Such a path is sometimes called a {\em feature line}. To formulate this constraint, suppose the feature line is given by adjacent points $A_1,A_2,\ldots,A_m$ on the triangular mesh where $A_i=(a_{i1},a_{i2},x_i)\in\RR^3$. For $A_iA_{i+1}$, the length of its ``shadow'' on the $xy$-plane is the euclidean distance
\begin{equation}\label{e:deltai}
\dd_i:=\|(a_{i+1,1},a_{i+1,2})-(a_{i,1},a_{i,2})\|=\sqrt{(a_{i+1,1}-a_{i,1})^2+(a_{i+1,2}-a_{i,2})^2}.
\end{equation}
Define also
\begin{equation}\label{e:ti}
t_1:=0\ ,\ t_2:=\dd_1\ ,\ 
t_3:=\dd_1+\dd_2\ ,\ \ldots\ ,\ t_m:=\dd_1+\cdots+\dd_{m-1}.
\end{equation}
Then the alignment constraint is written as
\begin{equation}\label{e:algn}
\forall i\in\{1,\ldots,\,m-2\}:\ \
(x_{i+1}-x_{i})/(t_{i+1}-t_i)
=
(x_{i+2}-x_{i+1})/(t_{i+2}-t_{i+1}).
\end{equation}

\begin{theorem}[projection onto an edge alignment constraint]
\label{t:algnRm}
Suppose the points $(a_{i1},a_{i2})\in\RR^2$ with ${i\in\{1,\ldots,m\}}$, forms a path in $\RR^2$. Let $\delta_i$ and $t_i$ be given respectively by \eqref{e:deltai} and \eqref{e:ti}. Let $F$ be the set of points $(x_1,\ldots,x_m)\in\RR^m$ such that \eqref{e:algn} is satisfied.

Let $z=(z_1,\ldots,z_m)\in\RR^m$.
Then the projection $\proj_F z\in\RR^m$ is given by
\begin{equation*}
\forall i\in\{1,\ldots,m\}:\ \left(\proj_F z\right)_i=f(t_i)=\alpha+\beta t_i.
\end{equation*}
where $f(t)=\alpha + \beta t$ is the least squares line for the points $(t_i,z_i)\in\RR^2$, ${i\in\{1,\ldots,m\}}$. 
\end{theorem}
\proof
First, $F$ is convex since all constraints in \eqref{e:algn} are linear. Next, we consider the points $(t_i,z_i)$ in $\RR^2$. The problem is to find $(x_1,\ldots,x_m)$ such that the points $(t_i,x_i)$ are aligned and
\begin{equation*}
\|x-z\|^2=\sum_{i=1}^m(x_i-z_i)^2
\quad\text{is minimized}.
\end{equation*}
This is the {\em least squares} problem for the points $(t_i,z_i)$. The proof is complete.
\endproof

\subsection{Surface Alignment Constraint}
\label{ss:prj_surf_algn}

The designer may want to patch several adjacent triangles on the mesh into a single polygon, in which case we say that these triangles are ``aligned". This is equivalent to requiring all vertices of the triangles to lie on the same plane. So we have the following result.

\begin{theorem}[projection onto a surface alignment constraint] Let $(a_{i1},a_{i2})$, $i\in\{1,\ldots,m\}$ be a collection of points in $\RR^2$ that are not on the same line. Let $F$ be the set of all points $(x_1,\ldots,x_m)\in\RR^{m}$ such that the points $\{(a_{i1},a_{i2},x_i)\}_{i\in\{1,\ldots,m\}}$
lie on the same plane in $\RR^3$.
Let $z=(z_1,\ldots,z_m)\in\RR^{m}$. Then the projection $\proj_Fz\in\RR^m$ is given by
\begin{equation*}
\forall i\in\{1,\ldots,m\}:\quad \big(\proj_F z\big)_i=f(a_{i1},a_{i2})=\alpha+\beta a_{i1}+\gamma a_{i2},
\end{equation*}
where $f(t_1,t_2)=\alpha+\beta t_{1}+\gamma t_{2}$ is the least squares plane for the points $(a_{i1},a_{i2},z_i)\in\RR^3$, ${i\in\{1,\ldots,m\}}$.
\end{theorem}
\proof Clearly, $F$ is a convex set. Let $x:=(x_1,\ldots,x_m)=\proj_F z$, then $x$ minimizes
\begin{equation*}
\|x-z\|^2=\sum_{i=1}^{m}|x_i-z_i|^2
\end{equation*}
subject to the constraint that $(a_{i1},a_{i2},x_i)$, ${i\in\{1,\ldots,m\}}$, lie on the same plane in $\RR^3$. Thus, it is equivalent to finding the least squares plane $f:\RR^2\to\RR$ for the points $(a_{i1},a_{i2},z_i)$, ${i\in\{1,\ldots,m\}}$, which has unique solution since these points are not on the same line. The conclusion follows.
\endproof

\section{Projections onto General Surface  Slope Constraints}
\label{s:prj_surf}

Surface slope constraints are any requirement imposed on the {\em slope} of a triangle. In this section, we provide a general setup for projections onto such constraints.

\subsection{Normal Vector and Surface Slope Constraint}
\label{ss:slConstr}

Let $(\bee_1,\bee_2,\bee_3)$ be the standard basis of $\RR^3$. Given three points $P_1=(p_{11},p_{12},h_1)$, $P_2=(p_{21},p_{22},h_2)$, and $P_3=(p_{31},p_{32},h_3)$ in $\RR^3$, the normal vector to the plane $P_1P_2P_3$ is the cross product
\begin{equation}\label{e:nt1t2}
\vec{n}
=\matrx{p_{11}-p_{31}\\ p_{12}-p_{32}\\ h_1-h_3}
\times
\matrx{p_{21}-p_{31}\\ p_{22}-p_{32}\\ h_2-h_3}
=:
\matrx{a_1\\ b_1\\ t_1}
\times
\matrx{a_2\\ b_2\\ t_2}
=\matrx{b_1 t_2-b_2 t_1\\
a_2 t_1-a_1 t_2\\
a_1b_2-a_2b_1},
\end{equation}
where the new variables $a_1,a_2,b_1,b_2,t_1,t_2$ are defined correspondingly, e.g., $a_1:=p_{11}-p_{31}$, etc. Also, we always assume that the ``shadows" of $P_1,P_2,P_3$ on $xy$-plane $(p_{11},p_{12})$, $(p_{21},p_{22})$, $(p_{31},p_{32})$ are not on the same line. Thus, $a_1b_2-a_2b_1\neq 0$.

So we can rescale and use
\begin{equation*}
\vec{n}=\left(\frac{b_1 t_2-b_2 t_1}{a_1b_2-a_2b_1},
-\frac{a_1 t_2-a_2 t_1}{a_1b_2-a_2b_1},1\right).
\end{equation*}
If we define
\begin{equation}\label{e:t12uv}
\matrx{t_1\\ t_2}=:\matrx{a_1 & b_1\\ a_2 & b_2}\matrx{u\\ v}
\quad\iff\quad
\matrx{u\\ v}:=\frac{1}{a_1b_2-a_2b_1}
\matrx{b_2 & -b_1\\ -a_2 & a_1}\matrx{t_1\\ t_2},
\end{equation}
then $\vec{n}=(-u,-v,1)$. Obviously, surface slope constraints depend solely on the normal vector $\vec{n}$. Thus, we can represent a surface slope constraint as
\begin{equation*}
g(u,v)\leq 0.
\end{equation*}
An important case of such constraints is the {\em surface orientation constraint} below.

\begin{definition}[surface orientation constraint]
\label{d:surorc}
Let $\Delta$ be a triangle with normal vector $\vec{n}=(-u,-v,1)\in\RR^3$ as above. Let $\vec{q}=(q_1,q_2,q_3)\in\RR^3\smallsetminus\{0\}$ be a unit vector and $\theta$ be an angle in $[0,\pi]$, the constraint
\begin{equation*}
\angle(\vec{n},\vec{q})\leq \theta,
\quad\text{or equivalently,}\quad
\cos\angle(\vec{n},\vec{q})
\geq
\cos\theta,
\end{equation*}
is called the {\em surface orientation constraint} specified by the pair $(\vec{q},\theta)$.
\end{definition}
In Section~\ref{ss:prj_sl_general}, we develop the general framework for projection onto surface orientation constraint. Then
in Sections~\ref{s:prj_maxsl} and \ref{s:prj_minsl} respectively, we will consider two special surface orientation constraints: the {\em surface maximum slope constraint} and the {\em surface oriented minimum slope constraint}.

\subsection{Projection onto a Surface Slope Constraint}
\label{ss:prj_sl_general}

Consider a single triangle determined by three points $Q_1=(p_{11},p_{12},w_1)$, $Q_2=(p_{21},p_{22},w_2)$, and $Q_3=(p_{31},p_{32},w_3)$ in $\RR^3$. Without loss of generality, we can assume 
\begin{equation*}
w_1+w_2+w_3=0.
\end{equation*} 
Projecting onto a slope constraint is to find the new heights $h_1,h_2,h_3$ that is a solution to the problem
\begin{align*}
\min_{(h_1,h_2,h_3)\in\RR^3}\quad
&\|(h_1,h_2,h_3)-(w_1,w_2,w_3)\|^2\\
\text{subject to}\quad&
\text{the triangle $P_1P_2P_3$ satisfies a given slope constraint,}\\ &\text{where $P_1=(p_{11},p_{12},h_1)$, $P_2=(p_{11},p_{12},h_2)$,
and $P_3=(p_{11},p_{12},h_3)$.}
\end{align*}
Defining $a_i,b_i,t_i,u,v$ as in \eqref{e:nt1t2} and \eqref{e:t12uv}, we have $h_1=a_1u+b_1v+h_3$ and $h_2=a_2u+b_2v+h_3$, i.e.,
\begin{equation}\label{e:h123uv}
\matrx{h_1\\ h_2\\ h_3}=\matrx{a_1 & b_1 & 1\\
a_2 & b_2 & 1\\
0  & 0 & 1}\matrx{u\\ v\\ h_3}.
\end{equation}
Suppose also the slope constraint on the triangle $P_1P_2P_3$ is written as $g(u,v)\leq 0$. Then the projection problem is converted to 
\begin{subequations}\label{e:160716c}
\begin{align}
\min_{(u,v,h_3)\in\RR^3}
\quad&
\phi(u,v,h_3):=
\left\|\matrx{a_1 & b_1 & 1\\
a_2 & b_2 & 1\\
0  & 0 & 1}\matrx{u\\ v\\ h_3}-\matrx{w_1\\ w_2\\ w_3}\right\|^2
\label{e:160716c-a}\\
\text{subject to}\quad&
g(u,v)\leq 0.\label{e:160716c-b}
\end{align}
\end{subequations}
Next, suppose further changing variables is necessary, for instance, $(u,v)$ is replaced by the new variables $(\wh{u},\wh{v})$ under a linear transformation
\begin{equation*}
\matrx{\wh{u}\\ \wh{v}}
:=Q\matrx{u\\ v},
\quad\text{where $Q$ is an invertible matrix}.
\end{equation*}
Then \eqref{e:160716c} is equivalent to
\begin{subequations}\label{e:160716h}
\begin{align}
\min_{(\wh{u},\wh{v},h_3)\in\RR^3}
\quad&
\left\|\matrx{\wh{a}_1 & \wh{b}_1 & 1\\
\wh{a}_2 & \wh{b}_2 & 1\\
0  & 0 & 1}\matrx{\wh{u}\\ \wh{v}\\ h_3}-\matrx{w_1\\ w_2\\ w_3}\right\|^2\\
\text{subject to}\quad&
\wh{g}(\wh{u},\wh{v}):=g(u,v)\leq 0,
\end{align}
\end{subequations}
where $\matrx{
\wh{a}_1 & \wh{b}_1\\ 
\wh{a}_2 & \wh{b}_2}:=\matrx{a_1 & b_1\\
a_2 & b_2}Q^{-1}$. That means we can treat \eqref{e:160716h} as \eqref{e:160716c} with {\em new} coefficients $\wh{a}_i,\wh{b}_i$ and constraint function $\wh{g}$.

Next, we will simplify the model \eqref{e:160716c} further.
Since \eqref{e:160716c-b} does not involve $h_3$, we can convert problem~\eqref{e:160716c} into two variables $(u,v)$ as follows: first, set the derivative of $\phi$ with respect to $h_3$ to zero
\begin{subequations}
\begin{align*}
\nabla_{h_3}\phi
&=2\matrx{1 & 1 & 1}\left[\matrx{a_1 & b_1 & 1\\
a_2 & b_2 & 1\\
0  & 0 & 1}\matrx{u\\ v\\ h_3}-\matrx{w_1\\ w_2\\ w_3}\right]\\
&=2\left[(a_1+a_2)u+(b_1+b_2)v-3h_3-(w_1+w_2+w_3)\right]=0.
\end{align*}
\end{subequations}
Since $w_1+w_2+w_3=0$, we obtain
\begin{equation}\label{e:h3uv}
h_3=-(1/3)\big[(a_1+a_2)u+(b_1+b_2)v\big].
\end{equation}
Substituting \eqref{e:h3uv} into \eqref{e:160716c-a}, we have
\begin{subequations}\label{e:190102a}
\begin{align}
\phi(u,v,h_3)&=\left\|
\matrx{a_1 & a_2 & 1\\
a_2 & b_2 & 1\\
0 & 0 & 1}
\matrx{1 & 0\\
0 & 1\\
-\frac{a_1+a_2}{3}& -\frac{b_1+b_2}{3}}
\matrx{u\\ v}
-\matrx{w_1\\ w_2\\ -(w_1+w_2)}\right\|^2\\
&=:\frac{1}{9}\left[
\frac{1}{2}\matrx{u & v}\matrx{A
& C\\
C & B}\matrx{u\\ v}-\matrx{w_a & w_b}\matrx{u\\ v}+Z\right]
\end{align}
\end{subequations}
where $Z$ is some constant independent of $(u,v)$ and
\begin{subequations}\label{e:ABCwab}
\begin{align}
A&:=2(a_1^2+a_2^2-a_1a_2)>0,\
B:=2(b_1^2+b_2^2-b_1b_2)>0, \label{e:ABCwab-a}\\
C&:=2a_1b_1+2a_2b_2-a_1b_2-a_2b_1,\quad\text{and}
\label{e:ABCwab-b}\\
\matrx{w_a & w_b}&:= 3\matrx{w_1 & w_2}\matrx{a_1 & b_1\\ a_2 & b_2}.
\end{align}
\end{subequations}
Thus, \eqref{e:160716c} is equivalent to the problem
\begin{subequations}\label{e:160716b}
\begin{align}
\min_{(u,v)\in\RR^2}\quad
&\varphi(u,v):=
\frac{1}{2}\matrx{u & v}\matrx{A & C\\ C& B}\matrx{u\\ v}-\matrx{w_a& w_b}\matrx{u\\ v}
\label{e:160716b-a}\\
\text{subject to}\quad
& g(u,v)\leq 0.
\label{e:160716b-b}
\end{align}
\end{subequations}
As long as we can find the solution $(u,v)$, we can obtain $(h_1,h_2,h_3)$ by using \eqref{e:h3uv} and \eqref{e:h123uv}.
In the case new variables $\wh{u},\wh{v}$ are used, we will use the corresponding coefficients $(\wh{a}_i,\wh{b}_i,\wh{u},\wh{v})$ in place of $(a_i,b_i,u,v)$. Finally, we show that $\varphi(u,v)$ is strictly convex.
\begin{lemma}\label{l:phi_scvx}
Suppose $a_1,a_2,b_1,b_2\in\RR$ such that $a_1b_2-a_2b_1\neq 0$.
Define $A,B,C$ by \eqref{e:ABCwab-a}--\eqref{e:ABCwab-b}. Then $\matrx{A & C\\ C& B}\succ 0$. Consequently, the function $\varphi(u,v)$ in problem \eqref{e:160716b} is strictly convex.
\end{lemma}
\begin{proof}
Note from \eqref{e:190102a} that $\matrx{A & C\\ C&B}=M^TM$ where
$M:=\matrx{a_1 & a_2 & 1\\
a_2 & b_2 & 1\\
0 & 0 & 1}
\matrx{1 & 0\\
0 & 1\\
-\frac{a_1+a_2}{3}& -\frac{b_1+b_2}{3}}$.
Since $a_1b_2-a_2b_1\neq 0$, $M$ has full column rank, which implies $M^TM$ is positive definite. It follows that $\varphi$ is strictly convex.
\end{proof}

\section{Projections onto Surface Maximum Slope Constraints}
\label{s:prj_maxsl}
Adopting the notation in Section~\ref{ss:slConstr}, we let $P_1P_2P_3$ be a triangle in $\RR^3$ with normal vector $\vec{n}=(-u,-v,1)$.
In certain cases, it is required that the angle between $\vec{n}$ and a given direction $\vec{q}$ must not exceed a given threshold. For example, suppose $P_1P_2P_3$ represents the desired ground, that cannot be too steep with respect to gravity, i.e., the slope of the plane $P_1P_2P_3$ must not exceed a threshold $s:=s_{\rm max}\in\RP$. Then the angle between $\vec{n}$ and $\bee_3:=(0,0,1)$ must satisfy $\angle(\vec{n},\bee_3)\leq \tan^{-1}(s)$, which is equivalent to
\begin{subequations}\label{e:maxsl-def2}
\begin{align}
\cos\angle(\vec{n},\bee_3)
=\frac{\scal{\vec{n}}{\bee_3}}{\|\vec{n}\|} \geq\cos(\tan^{-1}(s))=\frac{1}{\sqrt{1+s^2}}\\
\iff\quad
\frac{1}{\sqrt{u^2+v^2+1}}\geq\frac{1}{\sqrt{1+s^2}}
\quad\iff\quad 
u^2+v^2-s^2\leq 0.
\end{align}
\end{subequations}

\begin{definition}[surface maximum slope constraint]
We call \eqref{e:maxsl-def2} the {\em surface maximum slope constraint} with maximum slope $s$. This is a special case of surface orientation constraint where $(\vec{q},\theta)=(\bee_3,\tan^{-1}(s))$ (see~Section~\ref{ss:abs_prob}). 
\end{definition}
Using the general model \eqref{e:160716b}, we convert the projection onto surface maximum slope constraint to
\begin{subequations}\label{e:160716e}
\begin{align}
\min_{(u,v)\in\RR^2}\quad
&\varphi(u,v)
=\frac{1}{2}\matrx{u & v}\matrx{A & C\\ C& B}\matrx{u\\ v}-\matrx{w_a& w_b}\matrx{u\\ v}
\label{e:160716e-a}\\
\text{subject to}\quad
& u^2+v^2-s^2\leq 0,
\label{e:160716e-b}
\end{align}
\end{subequations}
where $A,B,C,w_a,w_b$ are given by \eqref{e:ABCwab}.
Rescaling by
$(u,v,w_a,w_b)
\longleftarrow
\Big(\frac{u}{s},\frac{v}{s},\frac{w_a}{s},\frac{w_b}{s}\Big)$,
we obtain an equivalent problem
\begin{subequations}\label{e:160716d}
\begin{align}
\min_{(u,v)\in\RR^2}\quad
&\varphi(u,v)= 
\frac{1}{2}\matrx{u & v}\matrx{A & C\\ C& B}\matrx{u\\ v}-\matrx{w_a& w_b}\matrx{u\\ v}
\label{e:160716d-a}\\
\text{subject to}\quad
& g(u,v):=u^2+v^2-1\leq 0.
\label{e:160716d-b}
\end{align}
\end{subequations}
This problem can be solved by several ways including numerical methods. For example, \eqref{e:160716d} is a special case of trust region subproblem which can be solved by means of generalized eigenvalue problems \cite{adachi17}.

As this is a projection problem that is needed in iterative optimization methods, it is desirable to have a closed form solution. Thus, in the rest of this section, we will aim to find such solution via Lagrange multipliers, also known as Karush-Kuhn-Tucker (KKT) conditions, see \cite{karush1939,kuhntucker1950} or \cite[Theorem~11.5]{beck14}, and Ferrari's method for quartic equations \cite{irving2013}.

First, due to Lemma~\ref{l:phi_scvx}, \eqref{e:160716d} is the problem of minimizing a strictly convex quadratic function $\varphi(u,v)$ over a closed, bounded, convex set in $\RR^2$. Thus, there exists a unique solution. To solve \eqref{e:160716d}, we start by finding the vertex $(u_0,v_0)$ of $\varphi(u,v)$, which is the unique solution of
\begin{equation*}
\nabla\varphi(u,v)=
\matrx{A & C\\ C& B}\matrx{u\\ v}-\matrx{w_a\\ w_b}=0.
\end{equation*}
Now we check if the vertex $(u_0,v_0)$ is inside or outside the feasible region:

{\em Case~1:} $g(u_0,v_0)\leq 0$ (inside). Then $(u_0,v_0)$ is the solution of \eqref{e:160716d}.

{\it Case~2:} $g(u_0,v_0)>0$ (outside). Then $\nabla\varphi(u,v)\neq 0$ for all $g(u,v)\leq 0$. Observe that for each value $\eta\geq 0$, the level set $\varphi(u,v)=\eta$ is an ellipse in $\RR^2$ whose center is the vertex $(u_0,v_0)$ (see Figure~\ref{fig:maxlvset}).
\begin{figure}[H]
\centering
\includegraphics[height=2.5in]{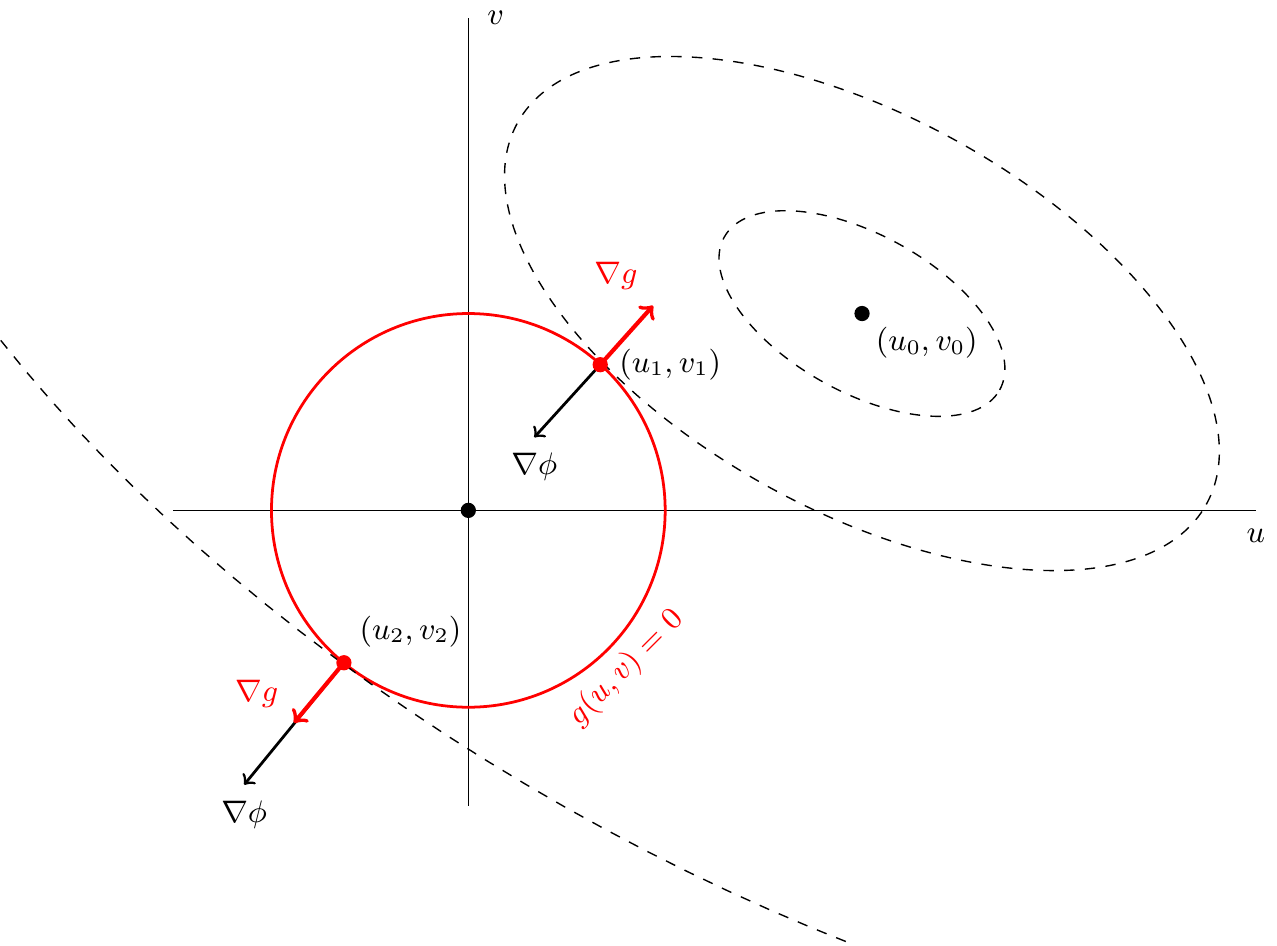}
\caption{Level sets of $\varphi(u,v)$ vs. feasible region.}
\label{fig:maxlvset}
\end{figure}
Hence, the tangent point of the unit circle $g(u,v)=0$ and the {\em smallest} elliptical level set of $\varphi$ that intersects the circle (see Figure~\ref{fig:maxlvset}), denoted by $(u_1,v_1)$, is the unique solution of the minimization problem \eqref{e:160716d}; and any tangent point of the unit circle $g(u,v)=0$ and the {\em largest} elliptical level set of $\varphi$ that intersects the circle (see Figure~\ref{fig:maxlvset}), denoted by $(u_2,v_2)$, is a solution of the maximization problem
\begin{equation*}
\max_{(u,v)\in\RR^2}\quad \varphi(u,v)
\quad\text{subject to}\quad g(u,v)\leq 0.
\end{equation*}
Based on this observation, there exist Lagrange multipliers $\lambda_1\geq 0$ and $\lambda_2\leq 0$ such that $(u_1,v_1,\lambda_1)$ and $(u_2,v_2,\lambda_2)$ satisfy the Lagrange multipliers system
\begin{subequations}\label{e:larg1}
\begin{align}
\nabla \varphi(u,v)+\tfrac{\lambda}{2}\nabla g(u,v)&=0,\label{e:larg1-a}\\
u^2+v^2-1&=0.\label{e:larg1-c}
\end{align}
\end{subequations}
Note that $\nabla\varphi(u_1,v_1)\neq 0$ and $\nabla\varphi(u_2,v_2)\neq 0$, so we must have $\lambda_1>0$ and $\lambda_2<0$. Moreover, since the minimization problem has a unique solution, we therefore conclude that {\em the system \eqref{e:larg1} must possess a unique solution $(u_1,v_1,\lambda_1)$ with $\lambda_1>0$, and at least a solution $(u_2,v_2,\lambda_2)$ with $\lambda_2<0$}.

In summary, {\em Case~2} reduces to finding the unique solution $(u,v,\lambda)$ of \eqref{e:larg1} with $\lambda>0$. First, \eqref{e:larg1-a} yields
\begin{equation}\label{e:uv2}
\matrx{A+\lambda & C\\ C & B+\lambda}\matrx{u\\ v}
=\matrx{w_a\\ w_b}.
\end{equation}
By Lemma~\ref{l:phi_scvx} and the assumption that $\lambda>0$, the matrix in \eqref{e:uv2} is positive definite. It follows that \eqref{e:uv2} has a unique solution
\begin{subequations}\label{e:uvlm}
\begin{align}
u&=\frac{w_a(B+\lambda)-w_bC}{(A+\lambda)(B+\lambda)-C^2}
=\frac{\lambda w_a +w_aB-w_bC}{\lambda^2+(A+B)\lambda+AB-C^2},\\
v&=\frac{w_b(A+\lambda)-w_aC}{(A+\lambda)(B+\lambda)-C^2}
=\frac{\lambda w_b +w_bA-w_aC}{\lambda^2+(A+B)\lambda+AB-C^2},
\end{align}
\end{subequations}
Next, substituting $u$ and $v$ into \eqref{e:larg1-c} yields
\begin{align*}
[\lambda^2+(A+B)\lambda+AB-C^2]^2
&=[w_a\lambda+(w_aB-w_bC)]^2+[w_b\lambda+(w_bA-w_aC)]^2\\
&=
(w_a^2+w_b^2)\lambda^2+2(w_a^2B+w_b^2A-2w_aw_bC)\lambda\\
&\quad +(w_aB-w_bC)^2+(w_bA-w_aC)^2.
\end{align*}
Defining the constants accordingly, we rewrite as
\begin{equation*}
(\lambda^2+R_1\lambda+R_2)^2=R_3\lambda^2+2R_4\lambda+R_5.
\end{equation*}

One can simply solve this equation by the classic Ferrari's method for quartic equations \cite{cardano1968}. Nevertheless, since there is a unique positive $\lambda$, we will find its explicit formula following Ferrari's technique. We introduce a real variable $y$
\begin{subequations}\label{e:qrtlm}
\begin{align}
(\lambda^2+R_1\lambda+R_2+y)^2
&=R_3\lambda^2+2R_4\lambda+R_5
+2(\lambda^2+R_1\lambda+R_2)y+y^2\\
&=(R_3+2y)\lambda^2
+2\big(R_4+R_1y\big)\lambda
+R_5+2R_2y+y^2.
\end{align}
\end{subequations}
We choose $y$ such that the right hand side is a perfect square in $\lambda$. Thus, the right hand side must have zero discriminant
\begin{align*}
&[R_4+R_1y]^2-(R_3+2y)(R_5+2R_2y+y^2)=0
\label{e:ycb-a}\\
\iff\quad &-2y^3+[R_1^2-R_3-4R_2]y^2+[2R_1R_4-2R_2R_3-2R_5]y
+R_4^2-R_3R_5=0.
\end{align*}
This is a cubic equation in $y$, so we use Cardano's method \cite{cardano1968} to find one real solution $y_0$. Then \eqref{e:qrtlm} becomes
\begin{equation}\label{e:qrtlm2}
\big(\lambda^2+R_1\lambda+R_2+y_0\big)^2=
(R_3+2y_0)\bigg(\lambda
+\frac{R_4+R_1y_0}{R_3+2y_0}\bigg)^2.
\end{equation}
As discussed above, this equation must have at least one positive and one negative solutions. Next, we solve for the (unique) positive $\lambda$:

{\em Case~2a:} $R_3+2y_0<0$. Then
$\lambda=-(R_4+R_1y_0)/(R_3+2y_0)$
must be the unique solution, which is a contradiction. Thus, this case cannot happen.

{\em Case~2b:} $R_3+2y_0=0$. Then $y_0=-R_3/2$ and
$\lambda^2+R_1\lambda+R_2+y_0=0$.
So
\begin{equation*}
\lambda
=\tfrac{1}{2}
\Big(-R_1\pm\sqrt{R_1^2-4R_2+2R_3}\Big).
\end{equation*}
Since there is only one positive $\lambda$, we take
$\lambda=
\big(-R_1+\sqrt{R_1^2-4R_2+2R_3}\big)/2$.

\smallskip
{\em Case~2c:} $R_3+2y_0>0$. Define $r:=\sqrt{R_3+2y_0}$.
Then \eqref{e:qrtlm2} becomes
\begin{equation*}
(\lambda^2+R_1\lambda+R_2+y_0)^2
=\Big(\lambda r
+\tfrac{R_4+R_1y_0}{r}\Big)^2.
\end{equation*}
This leads to two quadratic equations in $\lambda$
\begin{equation}\label{e:160925a}
\lambda^2+\left(R_1\pm r\right)\lambda
+\big(R_2+y_0
\pm\tfrac{R_4+R_1y_0}{r}\big)=0.
\end{equation}
Now if $R_4+R_1y_0=0$, then \eqref{e:160925a} consists of two quadratic equations that have constant term $R_2+y_0$. Thus, it yield an even number (possibly none) of positive solutions. Therefore, we must have $R_4+R_1y_0\neq 0$. Moreover, we must take the equation with negative constant term. So we set
$r\ \longleftarrow\
r\cdot\sgn(R_4+R_1y_0)$ and take only the equation
\begin{equation*}
\lambda^2+\left(R_1-r\right)\lambda
+\big(R_2+y_0
-\tfrac{R_4+R_1y_0}{r}\big)=0.
\end{equation*}
It follows that the positive $\lambda$ is
\begin{equation*}
\lambda=\tfrac{1}{2}\Big(r-R_1
+\sqrt{(r-R_1)^2-4\big(R_2+y_0
-\tfrac{R_4+R_1y_0}{r}\big)}\,\Big).
\end{equation*}
Next, we obtain $u$ and $v$ from  \eqref{e:uvlm} and then rescale variables $(u,v)\longleftarrow (su,sv)$. Finally, we obtain $(h_1,h_2,h_3)$ by using \eqref{e:h3uv} and \eqref{e:h123uv}.

\section{Projections onto Surface Oriented Minimum Slope Constraints}
\label{s:prj_minsl}

\subsection{Motivation from a Nonconvex Constraint}

Let $P_1P_2P_3$ be a triangle with the normal vector $\vec{n}=(-u,-v,1)$ as in Section~\ref{ss:slConstr}. In some cases, it is required that the angle $\angle(\vec{n},\vec{q})\geq \alpha$ for some given vector $\vec{q}$ and number $\alpha$. This happens, for example, if $P_1P_2P_3$ must have a slope at least $s:=s_{\rm min}\in\RP$. In this case, the angle between $\vec{n}$ and $\bee_3=(0,0,1)$ satisfies
\begin{subequations}\label{e:drain}
\begin{align}
&\angle{(\vec{n},\bee_3)}\geq \tan^{-1}(s)\label{e:drain-a}\\
\Leftrightarrow\quad
&\cos\angle(\vec{n},\bee_3)
={\scal{\vec{n}}{\bee_3}}/{\|\vec{n}\|} \leq\cos(\tan^{-1}(s))={1}/{\sqrt{1+s^2}}\\
\Leftrightarrow\quad
&\frac{1}{\sqrt{u^2+v^2+1}}\leq\frac{1}{\sqrt{1+s^2}}
\quad\Leftrightarrow\quad u^2+v^2-s^2\geq 0.
\end{align}
\end{subequations}
We refer to \eqref{e:drain} as the {\em surface minimum slope constraint} with minimum slope $s$. This is clearly a {\em nonconvex} constraint in $(u,v)$. Despite the projection onto this constraint is still available, nonconvexity may prevent iterative methods from convergence. It is worth to mention that similar minimum slope constraints are also critical in road design problem \cite{BK15}, which is again nonconvex and thus, somewhat hinders the theoretical analysis. Therefore, we will next present a novel idea to modify this constraint in a way such that minimum slope is preserved.

\subsection{The Surface Oriented Minimum Slope Constraint}

Condition \eqref{e:drain-a} implies the angle between $\vec{n}$ and the $xy$-plane satisfies
\begin{equation}\label{e:drain2}
\angle (\vec{n},\text{$xy$-plane})\leq
({\pi}/{2})-\tan^{-1}(s).
\end{equation} 
In some cases, it is reasonable to align the plane $P_1P_2P_3$ toward a given location/direction. For instance, in civil engineering drainage, the designer may want to direct the water to certain drains. Suppose we want $P_1P_2P_3$ inclined towards a unit direction $\vec{q}=(q_1,q_2,0)\in\RR^3$. Then we can fulfill \eqref{e:drain2} by requiring
$\angle{(\vec{n},\vec{q})}\leq \frac{\pi}{2}-\tan^{-1}(s)$,
i.e.,
\begin{equation*}
\cos\angle{(\vec{n},\vec{q})}
={\scal{\vec{n}}{\vec{q}}}/{\|\vec{n}\|}\geq \cos\big(\tfrac{\pi}{2}-\tan^{-1}(s)\big)={s}/{\sqrt{1+s^2}},
\end{equation*}
Substituting $\vec{n}=(-u,-v,1)$ and $\vec{q}=(q_1,q_2,0)$, we obtain
$\frac{-q_1u-q_2v}{\sqrt{u^2+v^2+1}}\geq \frac{s}{\sqrt{1+s^2}}$,
which is equivalent to
\begin{equation}\label{e:minsl}
q_1u+q_2v<0,\quad
u^2+v^2+1-\tfrac{1+s^2}{s^2}(q_1u+q_2v)^2\leq 0.
\end{equation}
Hence, we arrive at the following definition.
\begin{definition}[surface oriented minimum slope constraint]
\label{d:omsl}
We call \eqref{e:minsl} the {\em surface oriented minimum slope constraint} specified by $(\vec{q},s)$, where $\vec{q}=(q_1,q_2,0)\in\RR^3$ is a unit vector and $s\in\RPP$ is the minimum slope.
\end{definition}

Definition~\ref{d:omsl} is a special case of the {\em surface orientation constraint} in Section~\ref{ss:abs_prob} where $(\vec{q},\theta)=(\vec{q},\tfrac{\pi}{2}-\tan^{-1}(s))$.
It is worth mentioning that the constraint \eqref{e:minsl} not only guarantees surface minimum slope but also generates a {\em convex} feasible set.

\subsection{The Projection onto \eqref{e:minsl}} 

By employing Section~\ref{ss:prj_sl_general}, the projection onto the surface oriented minimum slope constraint is given by the solution to the problem
\begin{align*}
\min_{(u,v)\in\RR^2}\quad
&(a_1u+b_1v+h_3-w_1)^2+(a_2u+b_2v+h_3-w_2)^2+(h_3-w_3)^2\\
\text{subject to}\quad &\eqref{e:minsl}.
\end{align*}
Again, we first simplify this problem. Define $Q:=\matrx{
q_1 & q_2\\ q_2 & -q_1}$ where $(q_1,q_2,0)$ is the unit direction vector that defines the constraint \eqref{e:minsl}. Then $Q^2=\Id$, which implies $Q=Q^T=Q^{-1}$. Next, we change variables
\begin{equation*}
\matrx{
\wh{u}\\ \wh{v}}
:= Q\matrx{u\\ v}
=\matrx{q_1 u+q_2 v\\ q_2 u-q_1 v}
\quad\Leftrightarrow\quad
\matrx{u\\ v}
=Q\matrx{\wh{u}\\ \wh{v}}
=\matrx{q_1 \wh{u}+q_2 \wh{v}\\
q_2 \wh{u}-q_1 \wh{v}}.
\end{equation*}
Then the second inequality in \eqref{e:minsl} becomes
\begin{equation*}
(q_1\wh{u}+q_2\wh{v})^2+(q_2\wh{u}-q_1\wh{v})^2+1
-\Big(1+\frac{1}{s^2}\Big)\wh{u}^2
=\Big(-\frac{1}{s^2}\Big)\wh{u}^2+\wh{v}^2+1
\leq 0.
\end{equation*}
Thus, \eqref{e:minsl} becomes
\begin{equation}\label{e:minsl2}
\wh{u}<0,\quad \wh{v}^2-\frac{\wh{u}^2}{s^2}+1\leq 0.
\end{equation}
Following Section~\ref{ss:prj_sl_general}, we change coefficients (without relabeling)
\begin{equation}\label{e:new_aibi}
\matrx{a_1 & a_2\\ b_1 & b_2}
\longleftarrow Q
\matrx{a_1 & a_2\\ b_1 & b_2},
\end{equation}
and obtain an equivalent problem
\begin{equation}\label{e:160714a}
\min_{(\wh{u},\wh{v})\in\RR^2}\quad
\frac{1}{2}\matrx{\wh{u} & \wh{v}}
\matrx{A & C\\ C & B}\matrx{\wh{u}\\ \wh{v}}
-\matrx{w_a & w_b}\matrx{\wh{u}\\ \wh{v}}
\quad\text{subject to}\quad\eqref{e:minsl2}.
\end{equation}
where $A,B,C,w_a,w_b$ are defined by \eqref{e:ABCwab} with the new coefficients $a_1,a_2,b_1,b_2$ from \eqref{e:new_aibi}. Next, we change variables by
$(u,v,A,B,w_b)
\longleftarrow
\Big(\frac{\wh{u}}{s},\wh{v},sA,\tfrac{B}{s},\frac{w_b}{s}\Big)$,
then \eqref{e:160714a} is equivalent to
\begin{subequations}\label{e:160614c}
	\begin{align}
	\min_{(u,v)\in\RR^2}\   
	\quad&\varphi(u,v)=
    \frac{1}{2}\matrx{u & v}
    \matrx{A & C\\ C & B}\matrx{u\\ v}
    -\matrx{w_a & w_b}\matrx{u\\ v}
	\label{e:160614c-a}\\
	\text{subject to}\quad&
	g_1(u,v):=u<0,\quad
	g_2(u,v):=v^2-u^2+1\leq0.\label{e:160614c-b}
	\end{align}
\end{subequations}
Since all matrices in \eqref{e:new_aibi} are nonsingular, the new coefficients $a_1,a_2,b_1,b_2$ satisfy $a_1b_2-a_2b_1\neq 0$, which implies that $\varphi(u,v)$ is strictly convex by Lemma~\ref{l:phi_scvx}. The feasible set \eqref{e:160614c-b} is the left half of a hyperbola, thus, also a convex region (see Figure~\ref{fig:minlvset}). Therefore, \eqref{e:160614c} is the problem of minimizing a strictly convex quadratic function over a closed convex region, which must yield a unique solution.

Similar to Section~\ref{s:prj_maxsl}, we now will present a way to solve \eqref{e:160614c} by Lagrange multipliers. First, we find the vertex $(u_0,v_0)$ of $\varphi(u,v)$, which is the unique solution of
\begin{equation*}
\nabla\varphi(u,v)=
\matrx{A & C\\ C& B}\matrx{u\\ v}-\matrx{w_a\\ w_b}=0.
\end{equation*}

{\em Case~1:} $(u_0,v_0)$ is feasible, i.e., $u_0<0$ and $g_2(u_0,v_0)\leq0$. Then clearly $(u_0,v_0)$ is the unique solution of \eqref{e:160614c}.

{\em Case~2:} $(u_0,v_0)$ is not feasible. Then the unique solution $(u_1,v_1)$ of \eqref{e:160614c} is the tangent point of the left branch hyperbola curve $C:=\menge{(u,v)}{u<0,\ v^2-u^2+1=0}$ and the smallest level set of $\varphi(u,v)$ that intersects $C$, which is an ellipse centered at $(u_0,v_0)$. Figure~\ref{fig:minlvset} illustrates two possible scenarios.
\begin{figure}[H]
\centering
    \begin{subfigure}{.49\textwidth}
    \centering
    \includegraphics[height=2.5in]{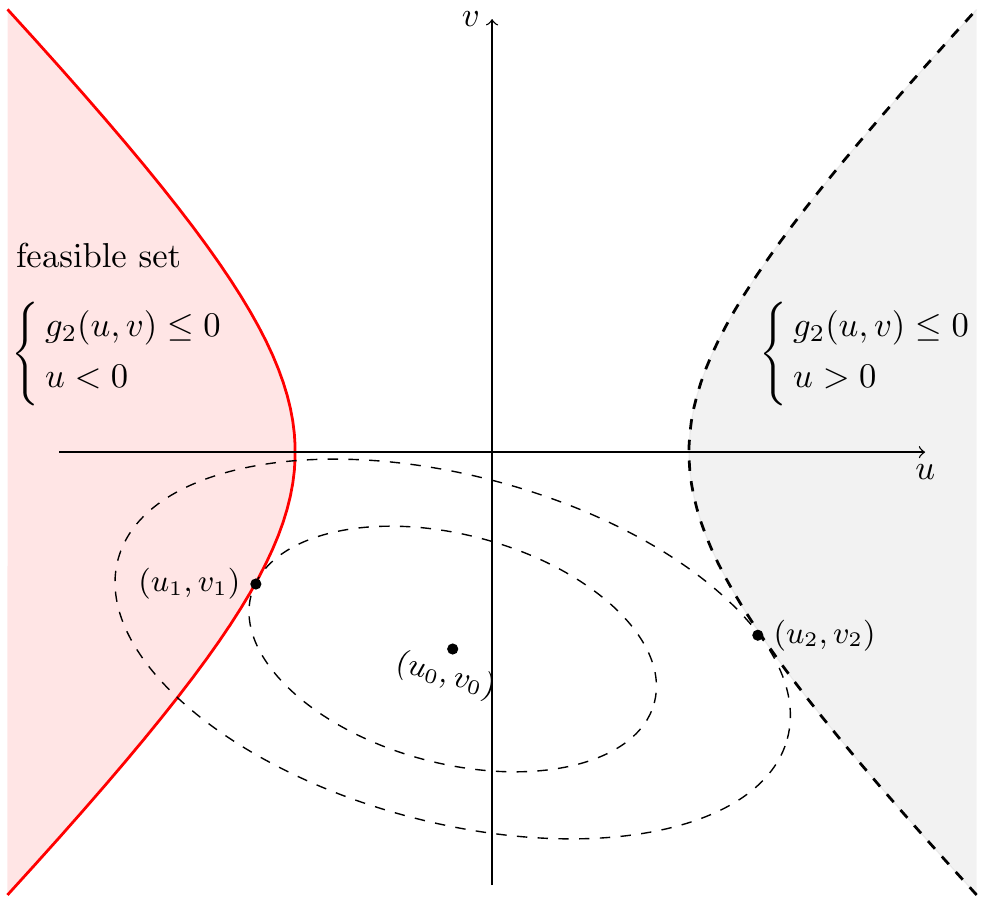}
    \caption{$g_2(u_0,v_0)>0$.}
    \label{fig:minlv_a}
    \end{subfigure}
    \begin{subfigure}{.49\textwidth}
    \centering
    \includegraphics[height=2.5in]{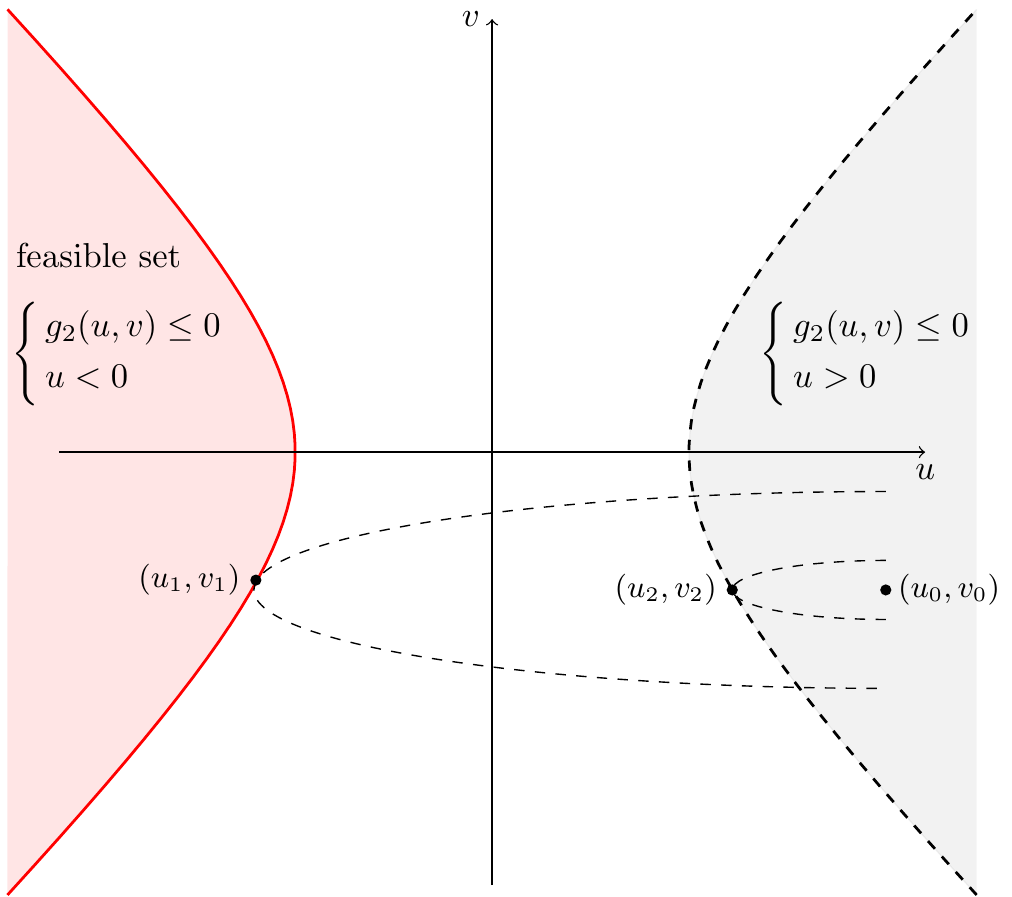}
    \caption{$g_2(u_0,v_0)\leq0$ and $u_0>0$.}
    \label{fig:minlv_b}
    \end{subfigure}
\caption{Level sets of $\varphi(u,v)$ vs. feasible region.}
\label{fig:minlvset}
\end{figure}
In both cases, we see that $g_1(u,v)=u<0$ is always an inactive constraint. Therefore, the Lagrange multipliers system reduces to
\begin{subequations}\label{e:larg2}
\begin{align}
\nabla\varphi(u,v)+\tfrac{\lambda}{2} \nabla g_2(u,v) &=0,\label{e:larg2-ab}\\
g_2(u,v)=v^2-u^2+1&=0.\label{e:larg2-c}
\end{align}
\end{subequations}
Utilizing Figure~\ref{fig:minlvset}, we conclude the following:

If $g_2(u_0,v_0)>0$ (see Figure~\ref{fig:minlv_a}), then \eqref{e:larg2} must have a unique solution $(u_1,v_1,\lambda_1)$ with $\lambda_1>0$ and $u_1<0$, and a unique solution $(u_2,v_2,\lambda_2)$ with $\lambda_2>0$ and $u_2>0$.

If $g_2(u_0,v_0)\leq 0$ and $g_1(u_0,v_0)=u_0>0$ (see Figure~\ref{fig:minlv_b}), then \eqref{e:larg2} must have a unique solution $(u_1,v_1,\lambda_1)$ with $\lambda_1>0$ and $u_1<0$. Also, \eqref{e:larg2} must have at least one solution $(u_2,v_2,\lambda_2)$ with $\lambda_2<0$.

In summary, {\em Case~2} reduces to finding the solution $(u_1,v_1,\lambda_1)$ of \eqref{e:larg2} where $\lambda_1>0$ and $u_1<0$. It then follows that $(u_1,v_1)$ is the unique solution of \eqref{e:160614c}. First, \eqref{e:larg2-ab} is
\begin{equation}\label{e:uv3}
\matrx{A-\lambda & C\\
C & B+\lambda}\matrx{u\\ v}=\matrx{w_a\\ w_b}.
\end{equation}
Define
\begin{subequations}\label{e:DDuDv}
\begin{align*}
D&:=(A-\lambda)(B+\lambda)-C^2
=-\lambda^2+(A-B)\lambda+(AB-C^2),\\
D_u&:=w_a(B+\lambda)-w_bC
=\lambda w_a + w_aB- w_b C,\\
D_v&:=w_b(A-\lambda)-w_a C
=-\lambda w_b + (w_b A- w_aC).
\end{align*}
\end{subequations} 
Suppose $D\neq 0$, then
$u=D_u/D$ and $v=D_v/D$. Substituting into \eqref{e:larg2-c} yields
\begin{subequations}
\begin{align*}
\Big[\lambda^2+(B- A)\lambda+ (C^2-AB)\Big]^2
&=\big(\lambda w_a + w_aB- w_bC\big)^2
-\big[\lambda w_b + (w_aC- w_b A)\big]^2\\
&=\Big( w_a^2-w_b^2\Big)\lambda^2
+2\Big(w_a^2 B+ w_b^2 A
-2 w_a w_b C\Big)\lambda\\
&\qquad+( w_a B- w_b C)^2
-( w_a C- w_b A)^2.
\end{align*}
\end{subequations}
By defining the constants accordingly, we rewrite the above identity as
\begin{equation*}
(\lambda^2+ R_1\lambda + R_2)^2= R_3\lambda^2 + 2 R_4\lambda + R_5.
\end{equation*}
Again, one can analyze this equation analogously to Section~\ref{s:prj_maxsl}. However, complication arises since there are possibly more than one positive $\lambda$'s. Instead, we expand
\begin{equation*}
\lambda^4+2R_1\lambda^3+(R_1^2+2R_2-R_3)\lambda^2
+2(R_1R_2-R_4)\lambda+R_2^2-R_5=0,
\end{equation*}
and use the classic Ferrari's method for quartic equation. Next, for each $\lambda>0$, we solve \eqref{e:uv3} as follows. 

\smallskip
{\em Case~2a:} $D\neq 0$. Then we simply compute
$u={D_u}/{D}$ and $v={D_v}/{D}$.

{\em Case~2b:} $D=0$. If either $D_u\neq0$ or $D_v\neq0$, then \eqref{e:uv3} has no solution $(u,v)$. So we now consider the remaining case that $D_u=D_v=0$. Then from \eqref{e:uv3} and \eqref{e:larg2-c}, we have
\begin{subequations}
\begin{align*}
Cu+(B+\lambda)v&= w_b,\\
v^2- u^2+1&=0.
\end{align*}
\end{subequations}
Since $B>0$ and $\lambda>0$, the first equation implies $\disp v={(w_b-Cu)}/{(B+\lambda)}$. So the second one becomes
$( w_b-Cu)^2-(B+\lambda)^2 u^2+(B+\lambda)^2=0$, i.e.,
\begin{equation}
\Big[C^2-(B+\lambda)^2\Big]u^2
-2w_bC u+ [w_b^2+(B+\lambda)^2]=0.\label{e:170316a}
\end{equation}
If the discriminant
$(w_bC)^2
-(C^2-(B+\lambda)^2)( w_b^2+(B+\lambda)^2)\geq 0$, then we obtain two solutions $u$. Since there is a unique pair $(u,v)$ with $u<0$, the quadratic equation \eqref{e:170316a} must yield one positive and one negative solutions
\begin{equation*}
u= \frac{w_bC\pm\sqrt{ (w_bC)^2
-(C^2-(B+\lambda)^2)( w_b^2+(B+\lambda)^2)}}
{C^2-(B+\lambda)^2},
\end{equation*}
and we must also have $C^2-(B+\lambda)^2<0$.
Therefore, the negative solution $u$ is
\begin{equation*}
u=\frac{w_bC+\sqrt{ (w_bC)^2
-(C^2-(B+\lambda)^2)( w_b^2+(B+\lambda)^2)}}
{C^2-(B+\lambda)^2}.
\end{equation*}

Next, among all $(u,v)$'s, choose the unique pair with $u<0$. Then rescale variables $u\longleftarrow su$. Finally, we obtain $(h_1,h_2,h_3)$ by \eqref{e:h3uv} and \eqref{e:h123uv}.

\section{Curvature Minimization}
\label{s:optim}

In some design problems, the designer may wish to construct a surface that is as ``smooth" as possible. This problem is referred to as minimizing the {\em curvature} between adjacent triangles in the mesh. In this section, we will address this problem.

\noindent
\begin{minipage}{.53\textwidth}
\begin{figure}[H]
\centering
\includegraphics[height=.8in]{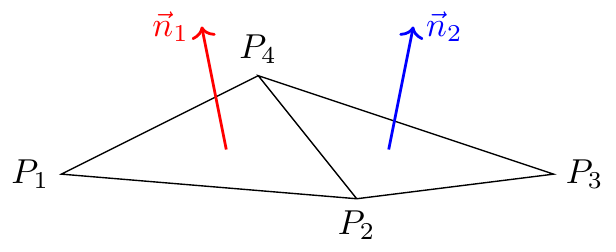}
\caption{Curvature between two triangles.}
\label{fg:ddn}
\end{figure}
\end{minipage}
\begin{minipage}{.44\textwidth}
Given the points $P_i=(p_{i1},p_{i2},h_i)$, $i=1,2,3,4$, so that they form two adjacent triangles $\Delta_1=P_1P_2P_4$ and $\Delta_2=P_2P_3P_4$ (see Figure~\ref{fg:ddn}).
\end{minipage}

Define \
$a_i:=p_{i1}-p_{41}$ and 
$b_i:=p_{i2}-p_{42}$, for $i=1,2,3$.
Then the respective normal vectors are
\begin{align*}
\vec{n}_{1}&=
\matrx{a_1\\ b_1\\ h_1-h_4}\times
\matrx{a_2\\ b_2\\ h_2-h_4}=
\matrx{-b_2h_1+b_1h_2+(b_2-b_1)h_4\\
a_2 h_1-a_1 h_2+(a_1-a_2)h_4\\
a_1b_2-a_2b_1},\\
\vec{n}_{2}&=
\matrx{a_2\\ b_2\\ h_2-h_4}\times
\matrx{a_3\\ b_3\\ h_3-h_4}=
\matrx{-b_3h_2+b_2h_3+(b_3-b_2)h_4\\
a_3 h_2-a_2 h_3+(a_2-a_3)h_4\\
a_2b_3-a_3b_2}.
\end{align*}
We rescale and obtain
\begin{align*}
\vec{n}_{1}&=
\bigg(\frac{-b_2h_1+b_1h_2+(b_2-b_1)h_4}{a_1b_2-a_2b_1},
\frac{a_2 h_1-a_1 h_2+(a_1-a_2)h_4}{a_1b_2-a_2b_1},
1\bigg),\\
\vec{n}_{2}&=
\Big(\frac{-b_3h_2+b_2h_3+(b_3-b_2)h_4}{a_2b_3-a_3b_2},
\frac{a_3 h_2-a_2 h_3+(a_2-a_3)h_4}{a_2b_3-a_3b_2},
1\Big).
\end{align*}
The curvature can be represented by the difference between $\vec{n}_1$ and $\vec{n}_2$, i.e.,
\begin{equation*}
\vec{\delta}_{12}:=\vec{\delta}_{\Delta_1\Delta_2}=\vec{n}_{1}-\vec{n}_{2}
=:\matrx{\scal{u}{h}\\ \scal{v}{h}},
\end{equation*}
where $h:=(h_1,h_2,h_3,h_4)$, $u:=(u_1,u_2,u_3,u_4)$, $v:=(v_1,v_2,v_3,v_4)$,
\begin{subequations}\label{e:uv1234}
\begin{alignat}{4}
u_1=-b_2/(a_1b_2-a_2b_1)\quad,
&&\quad u_2&=b_1/(a_1b_2-a_2b_1)+b_3/(a_2b_3-a_3b_2)\ , \\
u_3=-b_2/(a_2b_3-a_3b_2)\quad,
&&\quad u_4&=-(u_1+u_2+u_3)\ ,\\
v_1=a_2/(a_1b_2-a_2b_1)\quad,
&&\quad v_2&=-a_1/(a_1b_2-a_2b_1)-a_3/(a_2b_3-a_3b_2)\ , \\
v_3=a_2/(a_2b_3-a_3b_2)\quad,
&&\quad v_4&=-(v_1+v_2+v_3).
\end{alignat}
\end{subequations}
So for each pair $(\Delta_i,\Delta_j)$ of adjacent triangles, we find the corresponding vectors $h_{ij}=(h_1^{ij},h_2^{ij},h_3^{ij},h_4^{ij})\in\RR^4$, $u_{ij},v_{ij}\in\RR^4$, and compute the corresponding curvature
$\vec{\delta}_{ij}=\left(\scal{u_{ij}}{h_{ij}},\scal{v_{ij}}{h_{ij}}\right)$. We then aim to minimize all the curvatures between adjacent triangles. Thus, we arrive at the objective
\begin{equation*}
G_{1,*}(x):=\sum_{\text{all triangle pairs $(\Delta_i,\Delta_j)$}}\|\vec{\delta}_{ij}\|_*
\end{equation*}
where $\|\cdot\|_*$ can be either $1$-norm or max-norm in $\RR^2$. For $1$-norm, the objective is
\begin{equation}\label{e:170109a}
G_{1,1}(x)=\sum_{\text{all triangle pairs $(\Delta_i,\Delta_j)$}} |\scal{u_{ij}}{h_{ij}}|+|\scal{v_{ij}}{h_{ij}}|.
\end{equation}
For max-norm, the objective is
\begin{equation}\label{e:170109b}
G_{1,\infty}(x)=\sum_{\text{all triangle pairs $(\Delta_i,\Delta_j)$}} \max\{|\scal{u_{ij}}{h_{ij}}|,|\scal{v_{ij}}{h_{ij}}|\}.
\end{equation}

\begin{remark}[simplified computations for symmetric cases]
Suppose the two dimensional mesh satisfies the following symmetry: for every adjacent triangles $P_1P_2P_4$ and $P_2P_3P_4$, there exists $t\in\RR$ such that
\begin{equation*}
\overrightarrow{P_4P_1}+\overrightarrow{P_4P_3}=t\overrightarrow{P_4P_2}.
\end{equation*}
Then it follows that $(a_1,b_1)+(a_3,b_3)=t(a_2,b_2)$.
So we can deduce $a_1b_2-a_2b_1=a_2b_3-a_3b_2$. From \eqref{e:uv1234}, we have
\begin{equation*}
u=\frac{-b_2}{a_1b_2-a_2b_1}\big(1,-t,1,t-2\big)
\ \ \text{and}\ \
v=\frac{a_2}{a_1b_2-a_2b_1}\big(1,-t,1,t-2\big),
\end{equation*}
i.e., $u$ and $v$ are parallel. This simplifies the computations for \eqref{e:170109a} and \eqref{e:170109b}.
\end{remark}

Since our optimization methods require proximity operators, we will derive the necessary formulas. It is sometimes convenient to compute the proximity operator $\prox_f$ via the proximity operator of its {\em Fenchel conjugate} $f^*$, which is defined by
$f^*:X\to\RR: x^*\mapsto\sup_{x\in X}\big(\scal{x^*}{x}-f(x)\big)$. Indeed, if $\gamma>0$, then (see, e.g., \cite[Theorem~14.3(ii)]{BC2017})
\begin{equation*}
\forall x\in X:\quad
x=\prox_{\gamma f}(x)+\gamma\prox_{\gamma^{-1}f^*}(\gamma^{-1}x).
\end{equation*}
We also recall a useful formula from \cite[Lemma~2.3]{BKP16}: if $f:X\to\RR$ is convex and positively homogeneous, $\alpha>0$, $w\in X$, and
$h:X\to\RR:x\mapsto \alpha f(x-w)$, then
\begin{equation*}
\prox_{h}(x)= w+\alpha\prox_f\Big(\frac{x-w}{\alpha}\Big)
=x-\alpha\prox_{f^*}\Big(\frac{x-w}{\alpha}\Big).
\end{equation*}

\begin{theorem}\label{t:proxmax}
Let $\{u_i\}_{i\in I}$ be a system of finitely many vectors in $\RR^n$, and
\begin{equation*}
f:\RR^n\to\RR:x\to \max_{i\in I}\{|\scal{u_i}{x}|\}.
\end{equation*}
Then $f^*=\iota_{D}$ where
$D:=\conv\bigcup_{i\in I}\{u_i,-u_i\}$. Consequently, $\prox_{f}=\Id -\proj_D$.
\end{theorem}
\proof Suppose $x^*\in D$, then we can express
\begin{equation*}
x^*=\sum_{i\in I}\lambda_i u_i
\quad\text{where}\quad
\sum_{i\in I}|\lambda_i|\leq 1.
\end{equation*}
It follows that for all $x\in X$,
\begin{equation*}
\scal{x^*}{x}-f(x)=\sum_{i\in I}\lambda_i \scal{u_i}{x}-\max_{i\in I}|\scal{u_i}{x}|
\leq\Big(\sum_{i\in I}\lambda_i-1\Big)\max_{i\in I}|\scal{u_i}{x}|\leq 0.
\end{equation*}
So, $\disp f^*(x^*)=\sup_{x\in X}\big[\scal{x^*}{x}-f(x)\big]\leq0$. Notice that equality happens if we set $x=0$. Thus, $f^*(x^*)=0$.

Now suppose $x^*\not\in D$. Since $D$ is nonempty, closed, and convex, the classic separation theorem implies that there exists $x\in X$ such that
\begin{equation*}
\scal{x^*}{x}>\scal{u}{x}\quad\text{for all}\quad u\in D.
\end{equation*}
This leads to $\scal{x^*}{x}-f(x)>0$. Since $f$ is homogeneous, we have 
\begin{equation*}
f^*(x^*)\geq \scal{x^*}{\lambda x}-f(\lambda x)\to +\infty
\quad\text{as}\quad \lambda\to+\infty.
\end{equation*}
So $f^*(x^*)=+\infty$. Therefore, we can conclude that $f^*=\iota_D$.
\endproof

Finally, Examples~\ref{ex:seg} and \ref{ex:parlgrm} provide the necessary formulas to compute the proximity operators of the objectives in \eqref{e:170109a} and \eqref{e:170109b}, respectively.

\begin{example}
\label{ex:seg}
Given $\alpha>0$, a vector $u\in\RR^n\smallsetminus\{0\}$ and the function
\begin{equation*}
f:\RR^n\to\RR:x\mapsto\alpha|\scal{u}{x}|.
\end{equation*}
By Theorem~\ref{t:proxmax}, the proximity operator of $f$ is 
\begin{equation*}
\prox_f=\Id-\proj_D
\ \text{where}\ D:=[-\alpha u,\alpha u].
\end{equation*}
The explicit projection onto $D$ is given by (see \cite[Theorem~2.7]{BKP16})
\begin{equation*}
\proj_D x=
\min\Big\{1,\max\Big\{-1,\frac{\scal{\alpha u}{x}}{\|\alpha u\|^2}\Big\}\Big\}\alpha u
=\min\Big\{\alpha,\max\Big\{-\alpha,\frac{\scal{ u}{x}}{\|u\|^2}\Big\}\Big\}u.
\end{equation*}
\end{example}

\begin{example}
\label{ex:parlgrm}
Given two vectors $u,v\in\RR^n$ and
\begin{equation*}
f(x):\RR^n\to\RR:x\mapsto\max\{|\scal{u}{x}|,|\scal{v}{x}|\}.
\end{equation*}
By Theorem~\ref{t:proxmax}, the proximity operator of $f$ is 
\begin{equation*}
\prox_f=\Id-\proj_D
\ \text{where}\ 
D:=\conv\{u,v,-u,-v\}.
\end{equation*}
In general, $\prox_D$ is the projection onto a parallelogram in $\RR^n$.
\end{example}

\section{Experiments}
\label{s:experiment}

With the formulas for proximity and projection operators, we are ready to apply iterative methods to solve feasibility and optimization problems. In particular, we present an application to the civil engineer problem in Section~\ref{ss:civil} which is part of our motivation.

{\em Experiment setup:} In each of the three problems outlined in Figure~\ref{f:parking}, we aim to minimize the surface curvature using the objective function \eqref{e:170109a}, and subject to the requirements that: the maximum slope of all triangles does not exceed $4\%$; each triangle must incline toward its closest drain line (marked in blue) with minimum slope of $0.5\%$; and the triangle edges (marked in red) must be aligned. The DR algorithm \eqref{e:dr} will be used and it will stop when distance between two consecutive {\em governing} iterations are less than the tolerance $\varepsilon=0.001$ and the {\em monitored} iteration meets all constrained with (same) tolerance $\varepsilon$.

{\em Results:} In Figures~\ref{fig:parkingA}, \ref{fig:parkingB}, and \ref{fig:roundabout}, we show the solutions after various iterations of the DR algorithm. Triangles colored in red violate the design constraints (maximum and/or minimum slopes), whereas triangles in green satisfy the constraints. Below the surfaces are the contours, which become more regular with increasing iterations due to curvature minimization objective.

\begin{figure}[!htb]
    \centering
    \begin{subfigure}[b]{0.47\textwidth}
        \includegraphics[width=\textwidth]{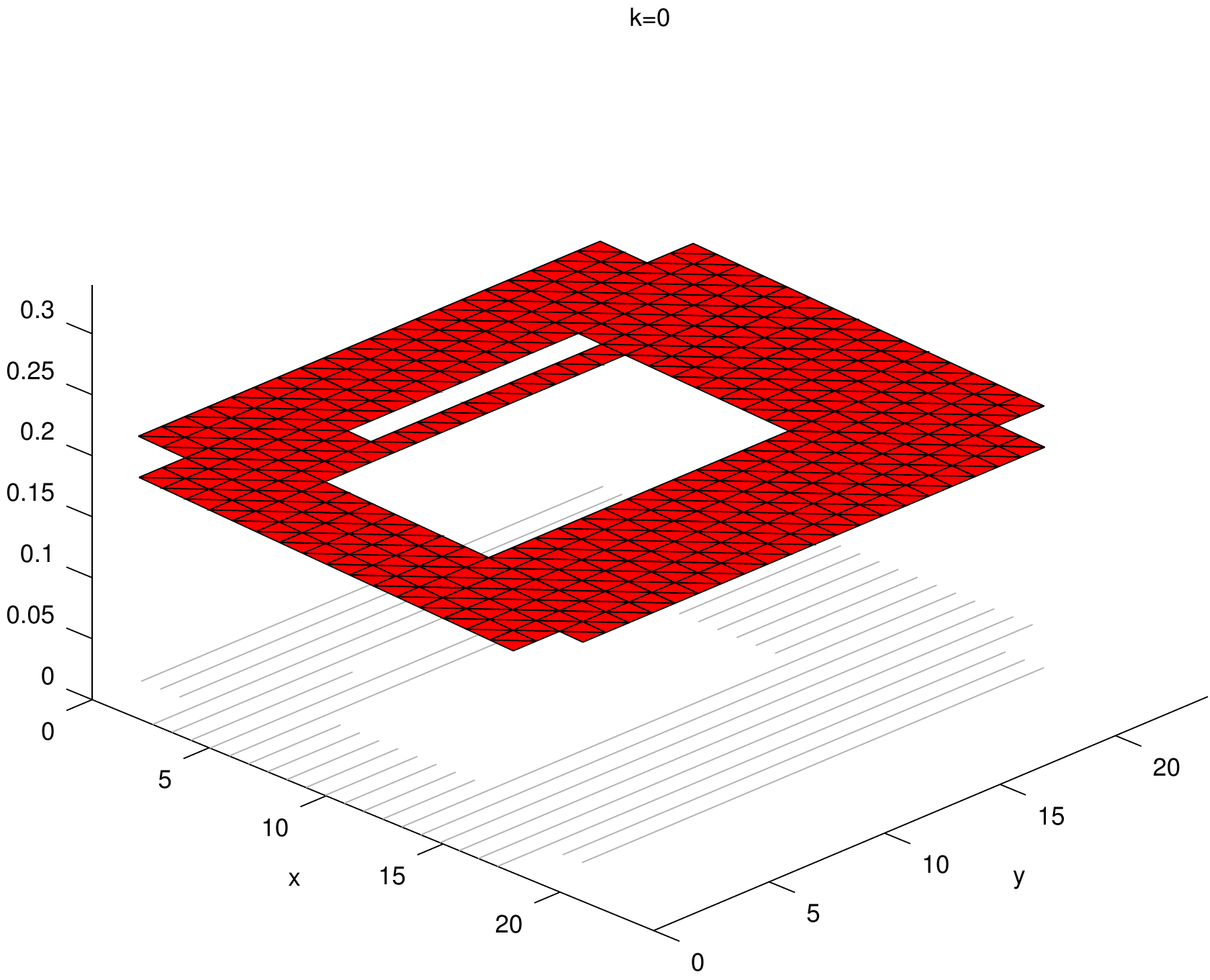}
        \caption{Starting conditions.}
        \label{fig:pAstart}
    \end{subfigure}
    ~ 
    \begin{subfigure}[b]{0.47\textwidth}
        \includegraphics[width=\textwidth]{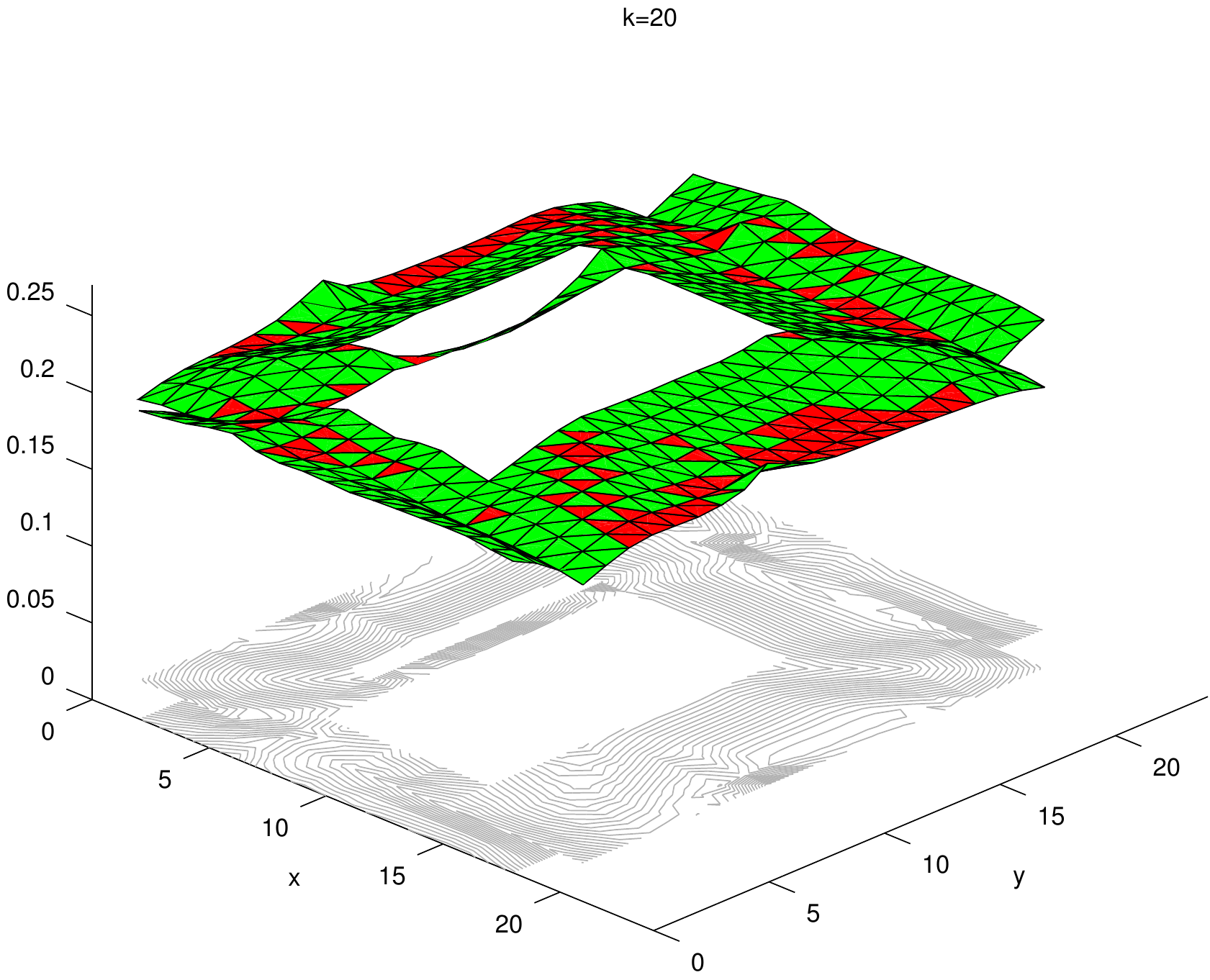}
        \caption{At $k=20$}
        \label{fig:pA10}
    \end{subfigure}
    
    \begin{subfigure}[b]{0.47\textwidth}
        \includegraphics[width=\textwidth]{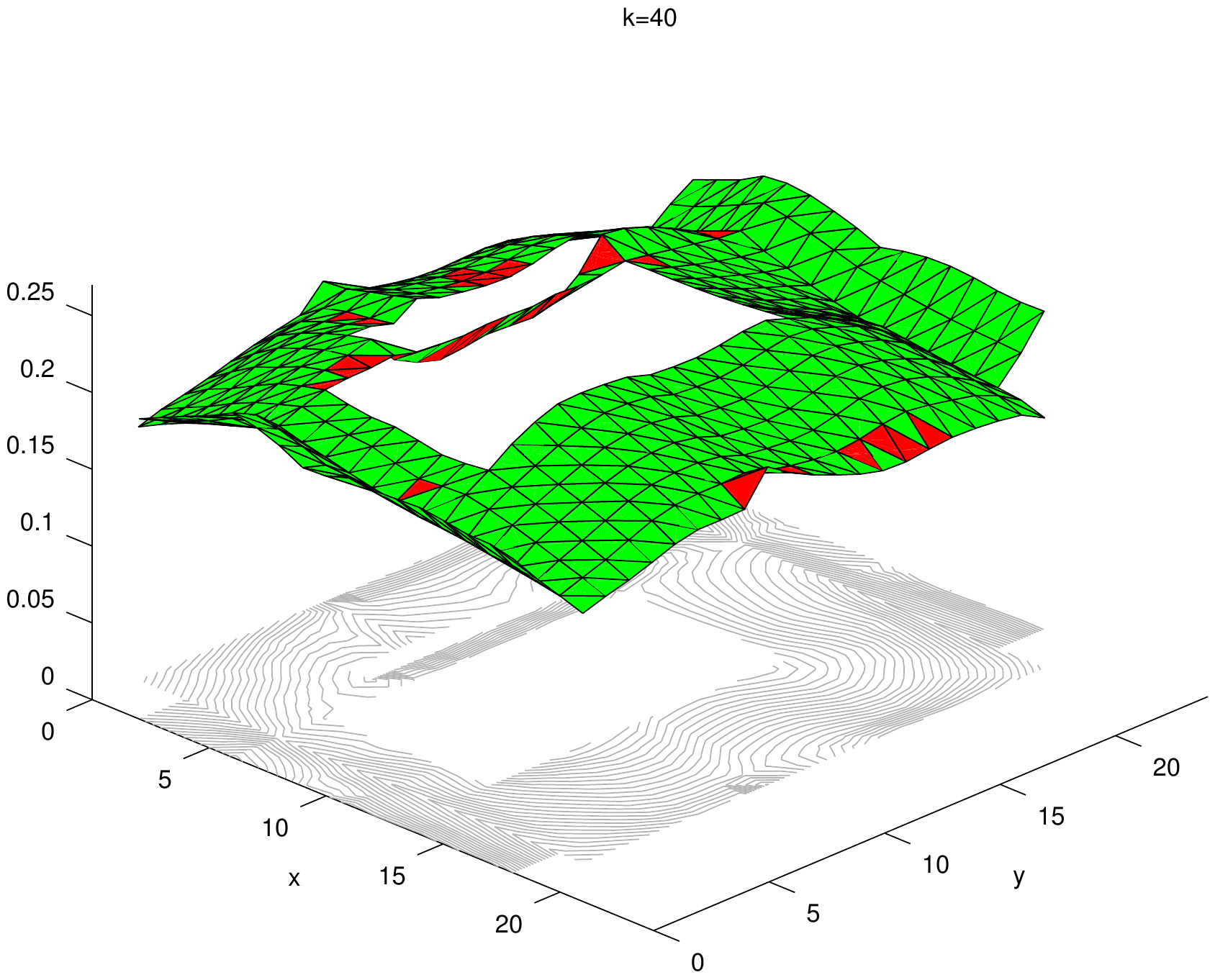}
        \caption{At $k=40$}
        \label{fig:pA20}
    \end{subfigure}
    ~ 
    \begin{subfigure}[b]{0.47\textwidth}
        \includegraphics[width=\textwidth]{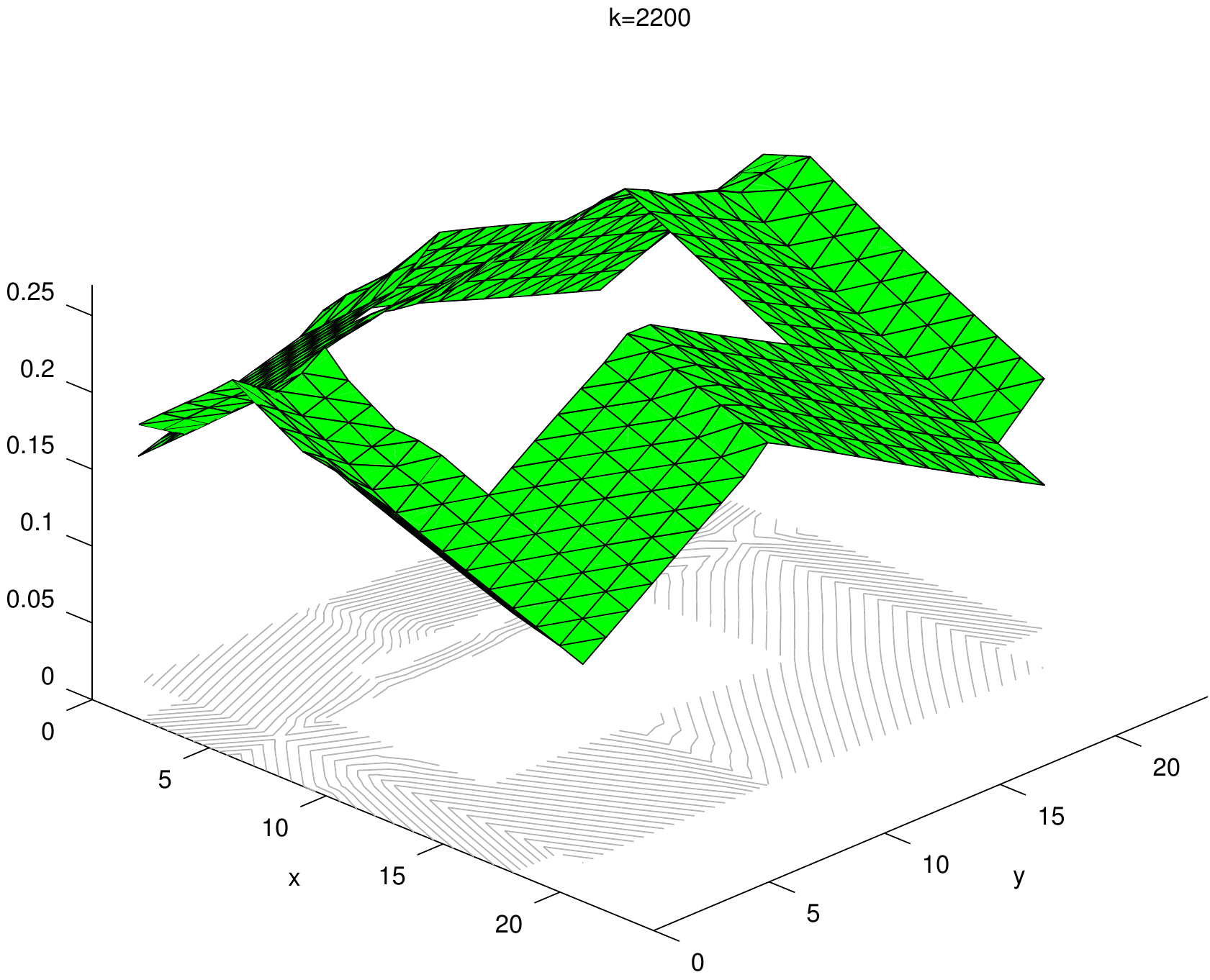}
        \caption{Final solution.}
        \label{fig:pAfinal}
    \end{subfigure}
    \caption{Grading design for a parking lot with corner drainage.}\label{fig:parkingA}
\end{figure}

\begin{figure}[!htb]
    \centering
    \begin{subfigure}[b]{0.47\textwidth}
        \includegraphics[width=\textwidth]{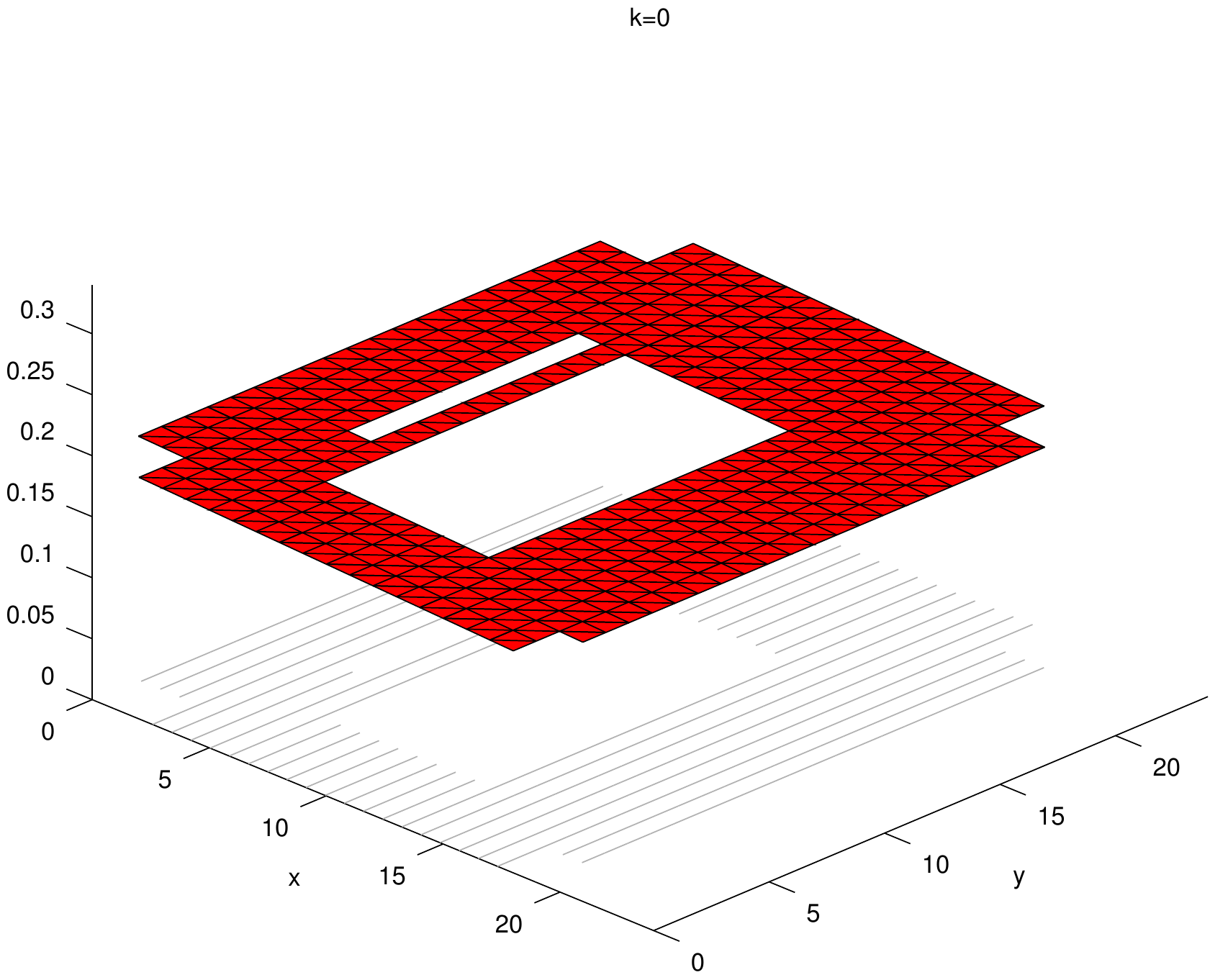}
        \caption{Starting conditions.}
        \label{fig:pBstart}
    \end{subfigure}
    ~ 
    \begin{subfigure}[b]{0.47\textwidth}
        \includegraphics[width=\textwidth]{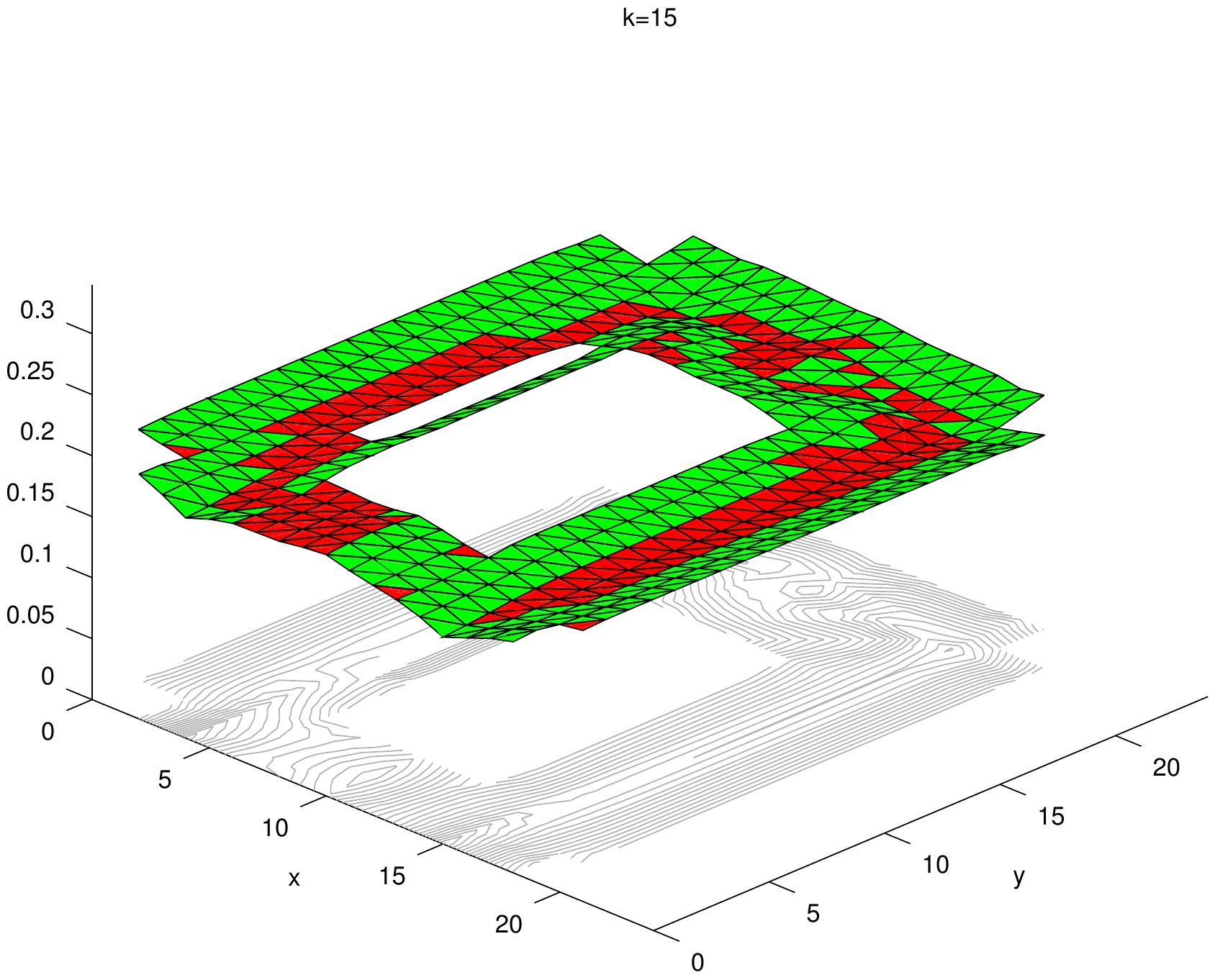}
        \caption{At $k=15$}
        \label{fig:pB10}
    \end{subfigure}
    
    \begin{subfigure}[b]{0.47\textwidth}
        \includegraphics[width=\textwidth]{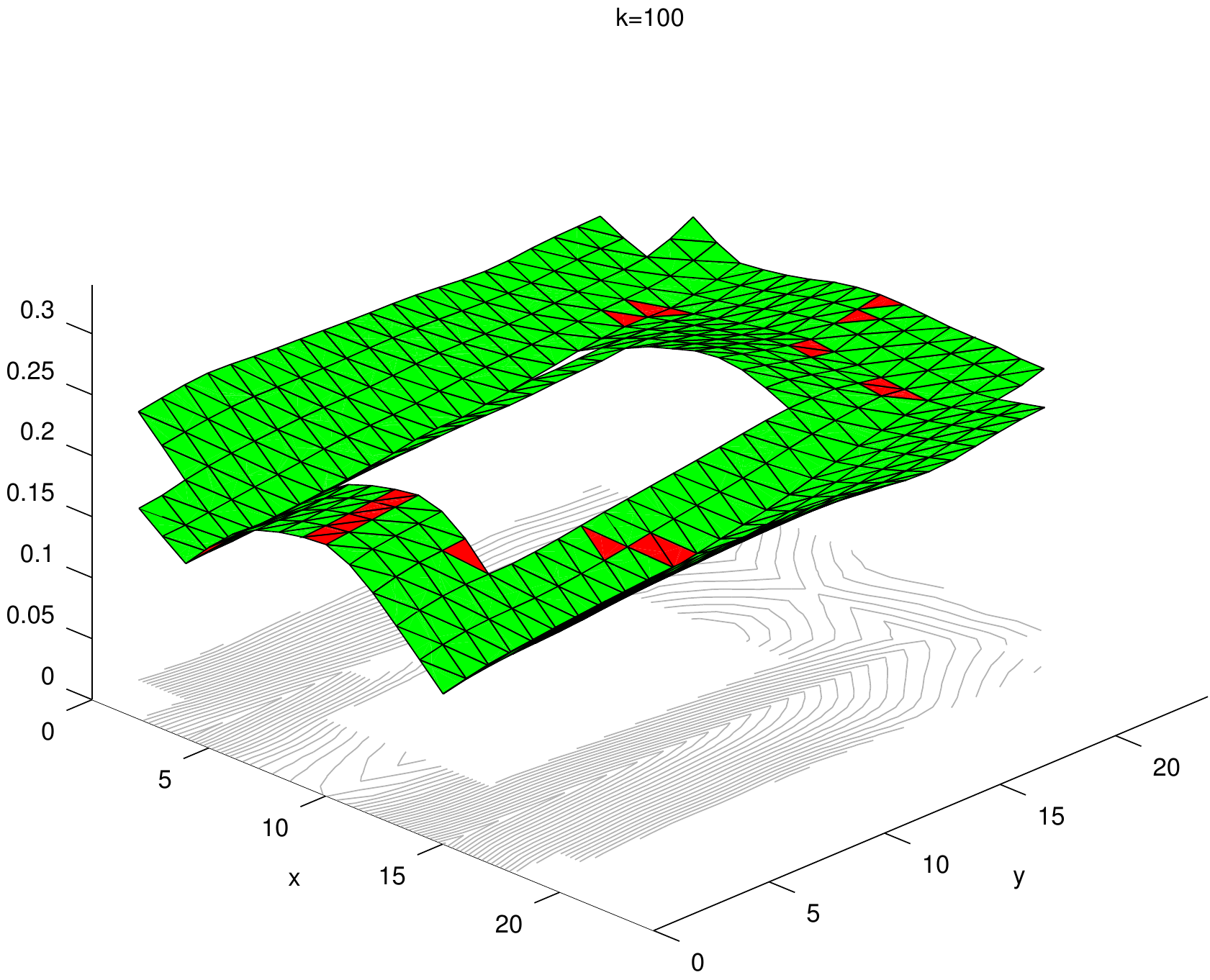}
        \caption{At $k=100$}
        \label{fig:pB20}
    \end{subfigure}
    ~ 
    \begin{subfigure}[b]{0.47\textwidth}
        \includegraphics[width=\textwidth]{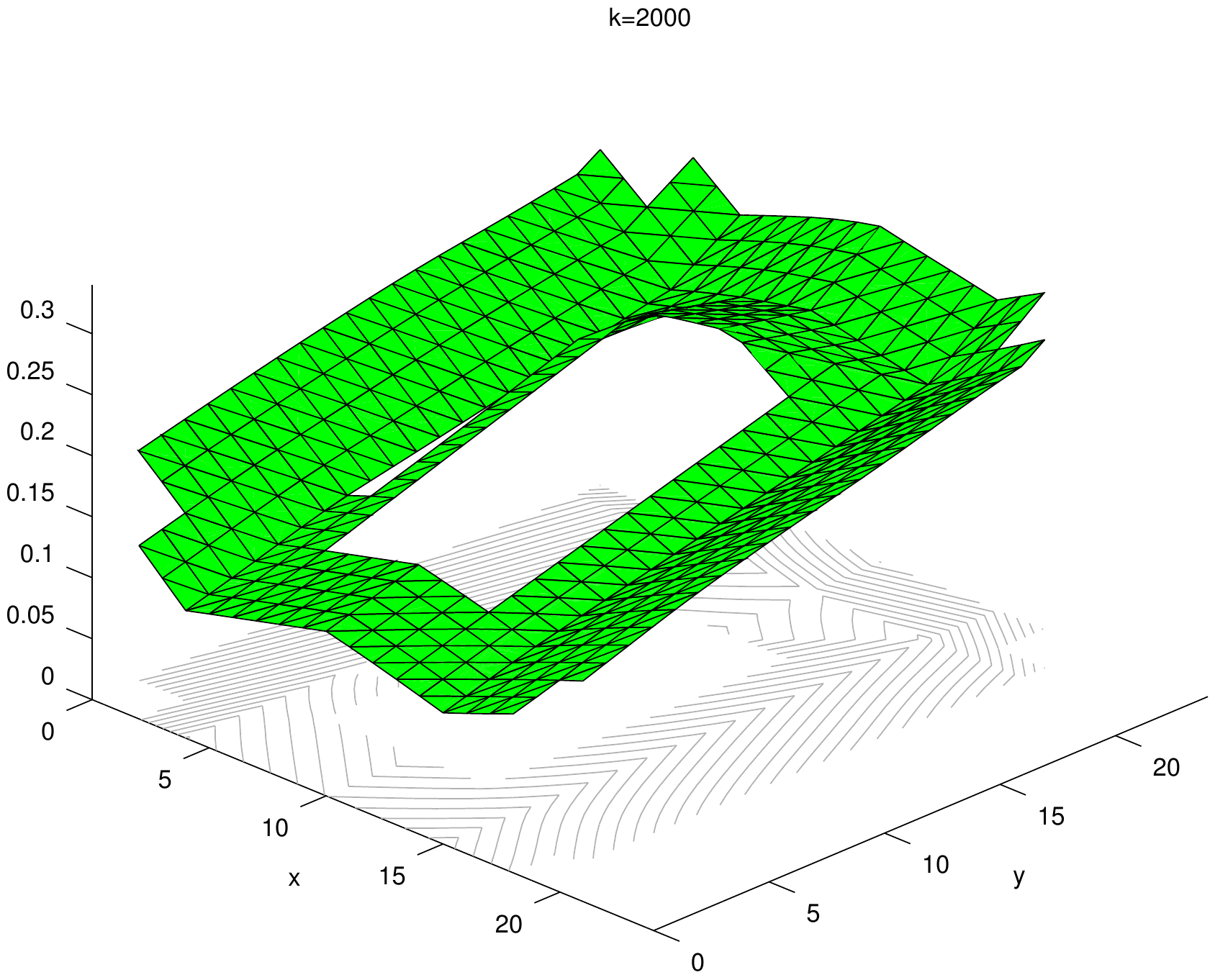}
        \caption{Final solution.}
        \label{fig:pBfinal}
    \end{subfigure}
    \caption{Grading design for a parking lot with side drainage.}\label{fig:parkingB}
\end{figure}

\begin{figure}[!htb]
    \centering
    \begin{subfigure}[b]{0.47\textwidth}
        \includegraphics[width=\textwidth]{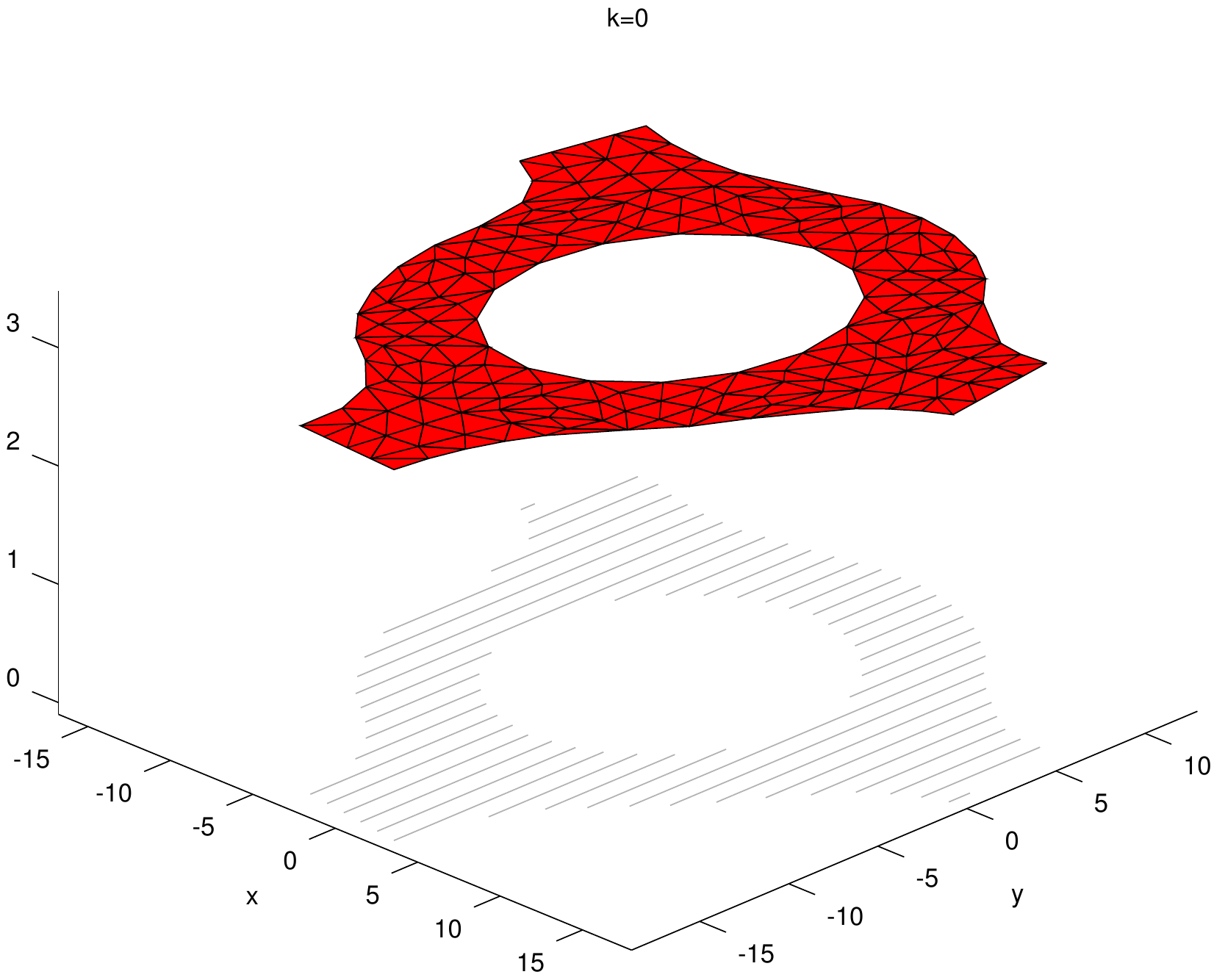}
        \caption{Starting conditions.}
        \label{fig:rstart}
    \end{subfigure}
    ~ 
    \begin{subfigure}[b]{0.47\textwidth}
        \includegraphics[width=\textwidth]{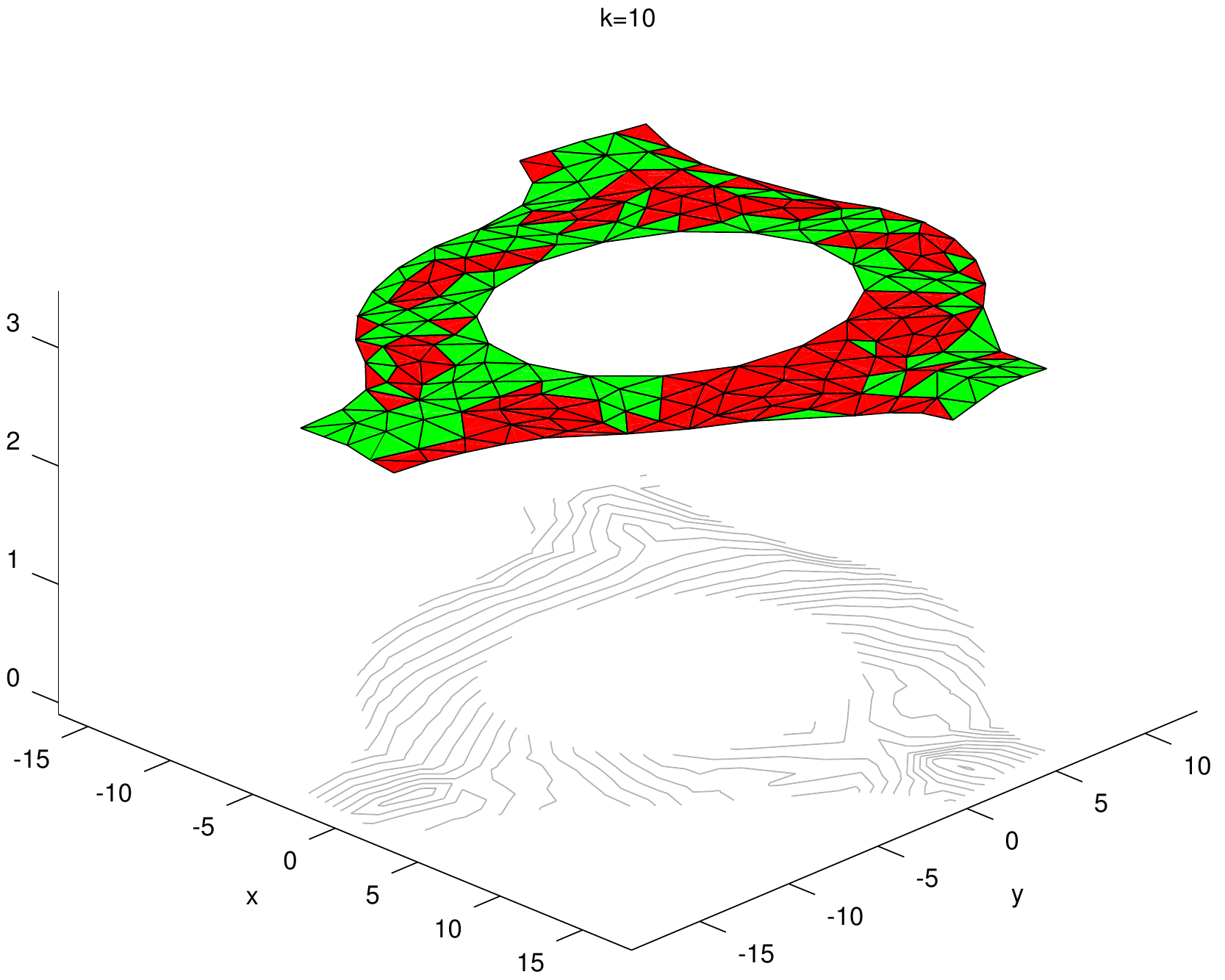}
        \caption{At $k=10$}
        \label{fig:r10}
    \end{subfigure}
    
    \begin{subfigure}[b]{0.47\textwidth}
        \includegraphics[width=\textwidth]{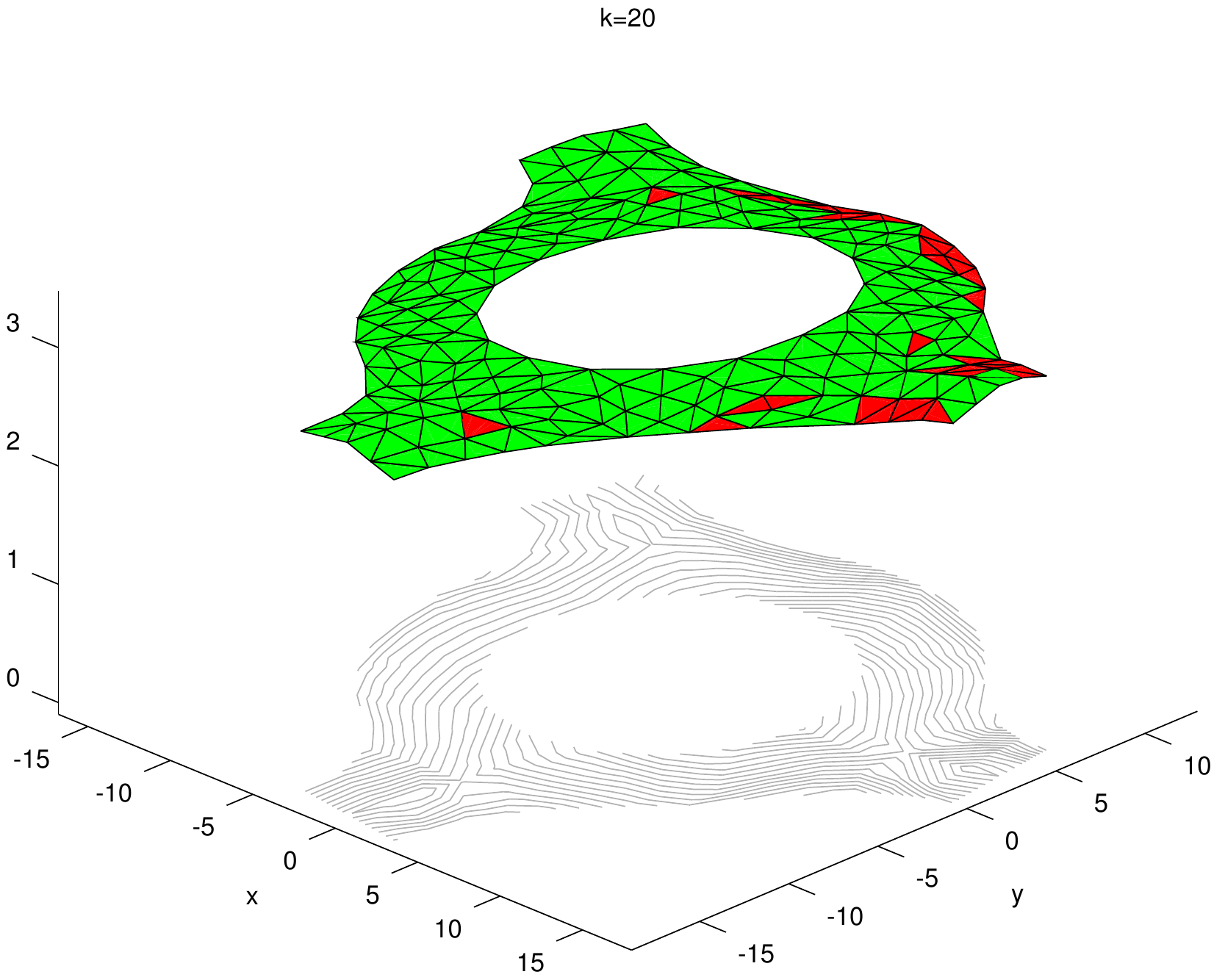}
        \caption{At $k=20$}
        \label{fig:r20}
    \end{subfigure}
    ~ 
    \begin{subfigure}[b]{0.47\textwidth}
        \includegraphics[width=\textwidth]{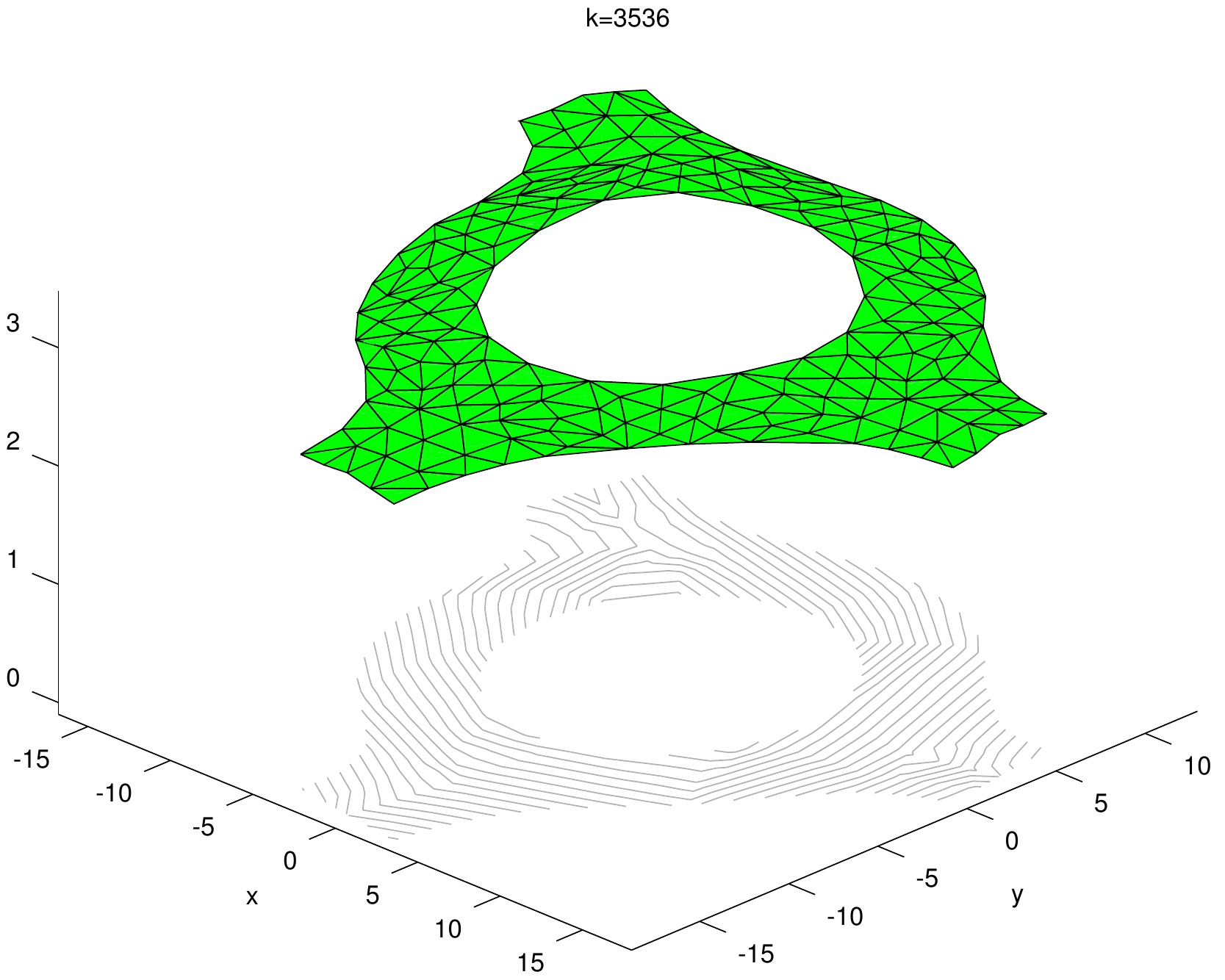}
        \caption{Final solution.}
        \label{fig:rfinal}
    \end{subfigure}
    \caption{Grading design for a roundabout.}\label{fig:roundabout}
\end{figure}

\section{Conclusion}
\label{s:conclus}
The manipulation of triangle meshes has many applications in computer graphics and computer-aided design. The paper presents a general framework for triangular design problems with spatial constraints. In particular, we model several important constraints and costs in suitable forms so that projection and proximity operators can be computed explicitly. With the help of iterative splitting methods, we are able to solve some complex design problems on these triangular meshes. Therefore, modeling constraints and their proximity operators, and using them in modern first-order optimization methods can be a successful approach to solve large-scale problems in industry and science.

\subsection*{Acknowledgement}

This research is partially supported by Autodesk, Inc. The authors are grateful to the Editors and two anonymous referees for their constructive suggestions that allow us to improve the original presentation.


\begin{thebibliography}{10}

\bibitem{adachi17}
S.\ Adachi, S.\ Iwata, Y.\ Nakatsukasa, and A.\ Takeda,
Solving the trust-region subproblem by a generalized eigenvalue problem,
{\em SIAM Journal on Optimization} 27 (2017), 269--291.

\bibitem{BBsirev93}
H.H.\ Bauschke and J.M.\ Borwein,
On projection algorithms for solving convex feasibility problems,
{\em SIAM Review} 38 (1996), 367--426.

\bibitem{BC2017}
H.H.\ Bauschke and P.L.\ Combettes,
{\em Convex analysis and monotone operator theory in {H}ilbert
  spaces}, second edition, CMS Books in Mathematics,
  Springer, New York (2017).

\bibitem{BK15}
H.H.\ Bauschke and V.R.\ Koch,
Projection methods: {S}wiss army knives for solving feasibility and best approximation problems with halfspaces,
{\em Infinite products of operators and their applications}, {\em Contemporary Mathematics} (2015) 636, 1--40. 

\bibitem{BKP16}
H.H.\ Bauschke, V.R.\ Koch, and H.M.\ Phan,
Stadium norm and Douglas-Rachford splitting: a new approach to
road design optimization,
{\em Operations Research} 64 (2016), 201--218.

\bibitem{beck14}
A.\ Beck, 
{\em Introduction to Nonlinear Optimization: Theory, Algorithms and Applications with Matlab}, MOS-SIAM Series on Optimization 19, Society for Industrial and Applied Mathematics (SIAM), Philadelphia, PA; Mathematical Optimization Society (MOS), Philadelphia, PA (2014). 

\bibitem{BCH2013}
R.I.\ Bo\c{t}, E.R.\ Csetnek, and A.\ Heinrich,
A primal-dual splitting algorithm for finding zeros of sums of maximal monotone operators,
{\em SIAM Journal on Optimization} 23 (2013), 2011--2036.

\bibitem{bricenocombettes2011}
L.M.\ Brice\~no Arias and P.L.\ Combettes,
A monotone + skew splitting model for composite monotone inclusions in duality,
{\em SIAM Journal on Optimization} 21 (2011), 1230--1250.

\bibitem{cardano1968}
G.\ Cardano,
{\em The great art or the rules of algebra},
The M.I.T. Press, Cambridge, Massa.-London (1968),
translated from the Latin and edited by T. Richard Witmer.

\bibitem{cegielskibook2012}
A.\ Cegielski,
{\em Iterative methods for fixed point problems in Hilbert spaces},
{Lecture Notes in Mathematics} 2057,
Springer, Heidelberg (2012).

\bibitem{cen_chen_com_12}
Y.\ Censor, W.\ Chen, P.L.\ Combettes, R.\ Davidi, and G.T.\ Herman,
On the effectiveness of projection methods for convex feasibility problems with linear inequality constraints,
{\em Computational Optimization and Applications} 
51 (2012), 1065--1088.

\bibitem{cen_zen_par1997}
Y.\ Censor and S.A.\ Zenios,
{\em Parallel optimization: Theory, algorithms, and applications},
Numerical Mathematics and Scientific Computation,
Oxford University Press, New York (1997).

\bibitem{comb-pesq-08}
P.L.\ Combettes and J.-C.\ Pesquet,
A proximal decomposition method for solving convex variational
inverse problems,
{\em Inverse Problems} 24 (2008), 065014 (27pp).

\bibitem{doug-rach-56}
J.\ Douglas and H.H.\ Rachford,
On the numerical solution of heat conduction problems in two and three space variables,
{\em Transactions of the American Mathematical Society}
82 (1956), 421--439.

\bibitem{irving2013}
R.\ Irving,
{\em Beyond the quadratic formula},
The Mathematical Association of America, Washington DC (2010).

\bibitem{karush1939}
W.\ Karush,
 {\em Minima of functions of several variables with inequalities as side conditions},
 ProQuest LLC, Ann Arbor, MI (1939),
 Thesis (SM)--The University of Chicago.

\bibitem{kuhntucker1950}
H.W.\ Kuhn and A.W.\ Tucker,
Nonlinear programming, {\em Proceedings of the Second Berkeley Symposium on Mathematical Statistics and Probability} (1950), 481--492, University of California Press, Berkeley and Los Angeles (1951).

\bibitem{lion-mercer-79}
P.-L.\ Lions and B.\ Mercier,
Splitting algorithms for the sum of two nonlinear operators,
{\em SIAM Journal on Numerical Analysis} 16 (1979), 964--979.

\bibitem{moreau65}
J.-J.\ Moreau,
Proximit\'e et dualit\'e dans un espace hilbertien,
{\em Bulletin de la Soci\'et\'e Math\'ematique de France},
93 (1965), 273--299.

\bibitem{rock70}
R.T.\ Rockafellar,
{\em Convex analysis},
Princeton University Press,
Princeton, NJ, (1970).

\bibitem{rusz2006}
A.\ Ruszczy\'nski,
 {\em Nonlinear optimization},
 Princeton University Press, Princeton, NJ (2006).

\end{thebibliography}
\end{document}